\documentclass[10pt]{article}
\RequirePackage[table]{xcolor}
\usepackage{authblk}
\usepackage{geometry}
\let\OLDthebibliography\thebibliography
\renewcommand\thebibliography[1]{
  \OLDthebibliography{#1}
  \setlength{\parskip}{0pt}
  \setlength{\itemsep}{0pt plus 0.6ex}
}

 \makeatletter
\renewcommand\AB@affilsepx{, \protect\Affilfont}

\usepackage{xcolor}

\usepackage{mathtools}
\usepackage{bm}
\usepackage{dsfont}
\usepackage{caption}
\usepackage{subcaption}
\usepackage[utf8]{inputenc}
\usepackage{graphicx} 
\usepackage{fullpage} 
\usepackage{amsmath,amssymb,amsthm} 
\usepackage{mathrsfs} 
\usepackage{verbatim} 
\usepackage{bbm}
\usepackage{xcolor}
\usepackage{marginnote} 

\usepackage[T1]{fontenc} 
\usepackage[english]{babel} 
\usepackage[plainpages=false,pdfpagelabels,colorlinks=true,linkcolor=red,urlcolor=red,citecolor=blue]{hyperref}

\def\bgam{\boldsymbol{\gamma}}
\numberwithin{equation}{section}
\newtheorem{theorem}{Theorem}[section]
\newtheorem{corollary}{Corollary}[section]
\newtheorem{lemma}{Lemma}[section]
\newtheorem{assumption}{Assumption}

\newtheorem{definition}{Definition}[section]

\newtheorem{remark}{Remark}[section]

\def\ddefloop#1{\ifx\ddefloop#1\else\ddef{#1}\expandafter\ddefloop\fi}
\def\ddef#1{\expandafter\def\csname c#1\endcsname{\ensuremath{\mathcal{#1}}}}
\ddefloop ABCDEFGHIJKLMNOPQRSTUVWXYZ\ddefloop
\def\ddef#1{\expandafter\def\csname s#1\endcsname{\ensuremath{\mathsf{#1}}}}
\ddefloop ABCDEFGHIJKLMNOPQRSTUVWXYZ\ddefloop
\def\ddef#1{\expandafter\def\csname b#1\endcsname{\ensuremath{\mathbf{#1}}}}
\ddefloop ABCDEFGHIJKLMNOPQRSTUVWXYZ\ddefloop
\def\ddef#1{\expandafter\def\csname b#1\endcsname{\ensuremath{\mathbf{#1}}}}
\ddefloop abcdefghijklmnopqrstuvwxyz\ddefloop
\def\argmin{\operatornamewithlimits{arg\,min}}

\def\E{\mathop{\mathbf{E}}}

\usepackage{authblk}

\title{\textbf{Universality in block dependent linear models with applications to nonparametric regression}}

\author[*,a]{Samriddha Lahiry}
\author[*,b]{Pragya Sur}
\affil[*]{\small{\textit{Department of Statistics, Harvard University}}}
\affil[a]{\href{mailto:slahiry@fas.harvard.edu}{\textit{slahiry@fas.harvard.edu}}}
\affil[b]{\href{mailto:pragya@fas.harvard.edu}{\textit{pragya@fas.harvard.edu}}}

\begin{document}

\maketitle


\renewcommand\Authands{ and }

\maketitle

\vspace{0.1in}









\begin{abstract} ~
Over the past decade, characterizing the exact asymptotic risk of regularized estimators in high-dimensional regression has emerged as a popular line of work.  This literature considers the proportional asymptotics framework, where the number of features and samples both diverge, at a rate proportional to each other. Substantial work in this area relies on Gaussianity assumptions on the observed covariates. Further, these studies often assume the design entries to be independent and identically distributed. Parallel research investigates the universality of these findings, revealing that results based on the i.i.d.~Gaussian assumption extend to a broad class of designs, such as i.i.d.~sub-Gaussians. However, universality results examining dependent covariates so far focused on correlation-based dependence or a highly structured form of dependence, as permitted by right rotationally invariant designs. In this paper, we break this barrier and study a dependence structure that in general falls outside the purview of these established classes.  We seek to pin down the extent to which results based on i.i.d.~Gaussian assumptions persist. We identify a class of designs characterized by a block dependence structure that ensures the universality of i.i.d.~Gaussian-based results. We establish that the optimal values of the regularized empirical risk and the risk associated with convex regularized estimators, such as the Lasso and ridge, converge to the same limit under block dependent designs as they do for i.i.d.~Gaussian entry designs. 
Our dependence structure differs significantly from correlation-based dependence, and enables, for the first time, asymptotically exact risk characterization in prevalent nonparametric regression problems in high dimensions. Finally,  we illustrate through experiments that this universality becomes evident quite early, even for relatively moderate sample sizes.\end{abstract}





\section{Introduction\label{sec_intro}}

Over the past decade, a novel regime to exploring high-dimensional asymptotics has emerged in supervised learning. This paradigm posits that the number of features ($p$) and the number of samples ($n$) both diverge with the ratio $p/n$ converging to a positive constant. For natural reasons, this paradigm is referred to as the proportional asymptotics regime. Numerous parallel techniques have emerged to address statistical and probabilistic questions within this framework, including random matrix theory \cite{marchenko1967distribution}, approximate message passing theory \cite{DMM09,bayati_mont,LASSO_Gauss, rangan2011generalized,javanmard2013state,lnk_flra_survey,donoho2016high,Sur_2,barbier2019optimal,maleki20bridge,mondelli2021approximate,AMP_survey,li2023approximate}, the leave-one-out/cavity method \cite{mezard1987spin,talagrand2003spin,mezard2009information,elkaroui_pnas_1,elkaroui_pnas_2,elkaroui_ptrf,lnk_flra_survey,Sur_4,Sur_2,chen2021spectral}, convex Gaussian min-max theorem \cite{stojnic_1,stojnic_2,stojnic_3,CT_1,CT_3,CT_5}, etc. The widespread adoption of this paradigm in statistical research can, in part, be attributed to its precise asymptotic results, which showcase exceptional performance when tested on finite samples (cf.~\cite{Sur_1,Sur_2,Sur_4,Sur_3,sur18supp,lnk_flra_survey, AMP_survey,MonSen_survey,liang2022precise,jiang2022new} and references cited therein). Moreover, this paradigm liberates researchers from stringent sparsity assumptions on underlying signals, accommodating structured classes of both dense and sparse signals.

However, a limitation arises within this paradigm—conventional tools often assume Gaussianity in the distribution of observed covariates. Extensive prior literature \cite{chatterjee2006generalization,erdHos2009local,tao2014random,Erdős_2011_survey,  BLM15,univ_elastic, chen_lam,hu2022universality,liang2022precise,panahi2017universal,gerace2022gaussian,DLS2,wang2022universality,han2023universality,saeed_mont,montanari2023universality,han2023distribution} suggests that this assumption is hardly a conceptual constraint. The results derived under this paradigm are often insensitive to the specific form of the covariate distribution, and rely more on the distribution's tails or on quantities such as the spectrum of the sample covariance matrix. 
In particular, prior work demonstrates the seamless extension of results derived under i.i.d.~Gaussian designs to i.i.d.~sub-Gaussian designs with suitable moments. Nevertheless, universality results under dependent design entries are relatively scarce. This paper addresses this challenge, revealing that under an appropriate notion of block dependence, universality results persist.

Our universality result substantially broadens the scope of statistical models considered within the proportional asymptotics regime. While the extension of universality beyond the independent model is intrinsically fascinating, we demonstrate that our results apply to widely used statistical models that had thus far remained outside the purview of this literature. Specifically, we leverage our universality result to provide precise risk characterization of estimators arising in nonparametric regression in the presence of high-dimensional covariates (Section \ref{sec_examples}). To the best of our knowledge, this class of problems had so far remained uncharted. Our result relies on a block dependence structure and we exhibit three popular examples where such dependence emerges naturally. The first example pertains to the additive model, well-studied in statistics, and applied in diverse fields such as biomedical research \cite{biomed_GAM_1,biomed_GAM_2}, agronomy \cite{Add_mod_agronomy}, environmental science \cite{air_poll_gam} etc. The second example, scalar- on-functional regression, arises from the expansive field of functional data analysis, motivated by longitudinal data, growth curves, time series, and more \cite{func_reg_3}. The third and final example is from genomics, particularly linkage-disequilibrium models \cite{Lewontin}, where block dependence structures arise naturally.

To introduce our setup formally, let $(Y_i,X_i)_{i=1}^{n}$ be i.i.d.~random variables with $Y_i\in \mathbb{R}$ and $X_i\in \mathbb{R}^p$.
We assume that $Y_i$ is a function of $X_i$ corrupted by an additive noise i.e.
$$Y_i=g^{*}(X_i)+\epsilon_i$$
with the most popular choice of $g^*$ being the linear model where one assumes $g^*(X_i)=X_i^T\beta^*$. 
In the next section, we will see that many nonparametric models can be reduced to linear models by choosing a suitable basis, and doing so induces interesting block dependent structures. To capture such nonparametric regression models, it suffices to focus on linear models with appropriate dependencies among covariates. This would be the main focus of our paper. 

We consider a penalized empirical risk minimization problem of the form $\min_{\beta} \hat{R}_n(Y,\mathcal{X},\beta)$ where 
\begin{equation}\hat{R}_n(Y,\mathcal{X},\beta)=\frac{1}{2n}\sum_{i=1}^n\|Y_i-\mathcal{X}^T_i\beta\|^2+\lambda f(\beta).
\label{emp_risk_min}\end{equation}
Here $\mathcal{X}=(\mathcal{X}_1,\ldots,\mathcal{X}_n)^T$, $Y=(Y_1,\ldots,Y_n)^T$, $X_i\in \mathbb{R}^p,\beta\in \mathbb{R}^p$ and $f$ is a convex penalty function. 
 
We denote the minimum and the minimizer of the empirical risk by \begin{equation}\hat{R}_n(Y,\mathcal{X})=\min_{\beta\in \mathbb{R}^p}\hat{R}_n(Y,\mathcal{X},\beta),\quad \text{ and }\hat{\beta}_{\mathcal{X}}\in \argmin_{\beta\in \mathbb{R}^p}\hat{R}_n(Y,\mathcal{X},\beta)\label{min_risk}\end{equation}
respectively. For the rest of this paper we will assume that the covariates are centred i.e. $E[\mathcal{X}_i]=0$.

Regularized estimators have been a subject of intense study over the decades. Extensive prior work characterizes the risk of such estimators in the proportional asymptotics regime when the covariates are drawn from a Gaussian distribution. To extend to non-Gaussian settings, the typical strategy involves proving  universality results that answer the following question: can rows of the data matrix $\mathcal{X}_i^T$
  be replaced by Gaussian vectors with matching moments and yield the same asymptotic results? 

Among existing universality results, \cite{korada_mont} established that when the design matrix  has i.i.d.~entries and finite sixth moment, the optimal value of the risk of a box constrained Lasso is the same as the optimal value when the entries are i.i.d.~Gaussian. Subsequently, \cite{univ_elastic} showed an analogous result in the context of elastic net. See also \cite{han2023universality} for a comprehensive set of universality results that hold when the design entries are i.i.d. Despite this stylized setting, results based on i.i.d.~Gaussian designs capture many new high-dimensional phenomena accurately in a qualitative sense. Further, this assumption allows one to characterize precise limits of objects of statistical relevance, such as risks of estimators  \cite{LASSO_Gauss,elkaroui_pnas_1,elkaroui_pnas_2,elkaroui_ptrf,donoho2016high,Sur_2},
 behaviour of finite-dimensional marginals \cite{Sur_2,Sur_4}, and prediction error of machine learning algorithms \cite{montanari2019generalization,deng2022model,liang2022precise,lss23}, as well as prove $\sqrt{n}$-consistency of estimators for average treatment effects \cite{jiang2022new}. In light of the elegant theory and novel high-dimensional phenomena this allows to capture, it is imperative to understand the limits of these results. In this paper, we seek to understand the extent to which results based on i.i.d.~Gaussian designs continue to hold in presence of dependence among the covariates.

In the context of dependent covariates, \cite{saeed_mont} established that  when rows of the design are of the form $\mathcal{X}_i=\Sigma^{1/2}Z_i$ with entries of $Z_i$ i.i.d.~subGaussian, the optimal value of the risk matches that of the design whose rows are of the form $\Sigma^{1/2}G_i$, where $G_i$ has i.i.d.~Gaussian entries. However, in a broad class of problems, the dependence structure among covariates may not be in the precise form $\Sigma^{1/2}Z_i$ with $Z_i$ containing i.i.d.~entries. On the other hand, \cite{DLS1,DLS2,wang2022universality} studied universality under curated forms of dependence, as allowed by right rotationally invariant designs, that are critical in compressed sensing applications \cite{donoho2009observed,monajemi2013deterministic,abbara2020universality}. See \cite{li2023spectrum} for detailed discussion and examples on the nature of dependent problems this class can handle. These forms of dependence studied in prior work fails to capture general sub-Gaussian designs that may arise naturally in important statistical models---this is the main focus of our paper. 

In our conquest to understand dependent covariates, we recognize that "too much" dependence among entries of $\mathcal{X}_i$ may break universality, as demonstrated in \cite{elkaroui_ptrf,han2023universality}. In the latter, the authors show that  for isotropic designs, if each entry in a row of the design depends on every other entry, the estimation risk is asymptotically different from the model with i.i.d.~Gaussian covariates.
However, noteworthy examples of dependence exist where each entry of the design depends on a subset of the entire row (cf.~Section \ref{sec_examples}). A natural question then arises: do universality results hold under such dependence structures? We answer in the affirmative, identifying a class of  dependence structures that ensure universality results continue to hold. In particular, we show the following\\

\textit{Let $\mathcal{X}$ be a design matrix with i.i.d.~mean zero rows.  Further assume that each row
$\mathcal{X}_i$ is ``block dependent'' and subGaussian. Let $\mathcal{G}$ be a design matrix with i.i.d.~rows  distributed as a multivariate Gaussian $\mathcal{N}(0,\Lambda)$ for a diagonal matrix $\Lambda$. If $E(\mathcal{X}_i\mathcal{X}^T_i)=\Lambda$ then we have}

$$\min_{\beta}\hat{R}_n(Y,\mathcal{X},\beta)\approx \min_{\beta}\hat{R}_n(Y,\mathcal{\mathcal{G}},\beta) \quad \text{and} \quad 
    \|\hat{\beta}_{\mathcal{X}}-\beta_{0}\|^2  \approx \|\hat{\beta}_{\mathcal{G}}-\beta_{0}\|^2.$$
\\
This simple mathematical result opens the avenue for exact risk characterization of regularized estimators in high-dimensional models that had so far evaded the literature (see Section \ref{sec_examples} for concrete examples). We define the notion of block dependence formally in Section \ref{sec_examples}. To the best of our knowledge, our result is the first in the literature on universality in high-dimensional regression that goes beyond  correlation-based dependence or right rotationally invariant designs. We provide a more detailed literature review  in Section \ref{related_lit}. For simplicity, we focus on two of the most commonly used penalty functions---the $\ell_1$ (Lasso) or the $\ell^2_2$ (ridge). These differ drastically in the geometry they induce in high dimensions. In this light, we believe that most other popular examples of separable penalties can be handled by extending our current proof techniques.

\subsection{Outline of the results}
To illustrate the importance of considering block dependent designs, we present various examples in Section \ref{sec_examples}. The initial example revolves around the commonly used additive model, where we demonstrate that by considering block dependent settings, one can extend results from parametric models to nonparametric models. Additionally, we explore an example of scalar-on-function regression where such dependence may emerge. The third example delves into a genomics model where block dependent designs naturally come into play. Section \ref{prelim} presents preliminaries and Section \ref{main_res} presents our main results. In Section \ref{sec_simul}, we present simulations corroborating our theory while Section \ref{related_lit} reviews literature related to our work.
Section \ref{sec_prf_otln} provides the proof of the main theorems and corollaries. Finally, we conclude in 
Section \ref{sec_disc} with a discussion on possible generalizations of our results, while the proofs of the auxiliary lemmas are given in the Appendix.

\subsection{Notations \label{sec_note}}
We use the $O(.)$ and $o(.)$ to denote asymptotically bounded and asymptotically goes to $0$ respectively i.e we say $a_n=O(\gamma_n)$ if $\limsup|a_n/\gamma_n|\leq C$ and $a_n=o(\gamma_n)$ if $\limsup|a_n/\gamma_n|=0$. Similarly $O_p(.)$ and $o_p(.)$ are reserved for the stochastic versions of the same quantities. We will use $\xrightarrow{P}$ to denote convergence in probability. The $\ell_1$ and $\ell_2$ norms of vectors are denoted by $\|.\|_1$ and $\|.\|$ respectively unless otherwise mentioned. For matrices the $\|A\|$ will always denote the largest singular value of $A$ (operator norm) and hence the same notation $\|.\|$ will be reserved for the  $\ell_2$ norm of a vector and the operator of norm of the matrix, the difference being understood from the context. We will denote the set $\{1,\ldots,p\}$ by $[p]$. For any subset $S\subset [p]$ and a matrix $X$ we will denote 
 the submatrix whose column indices are restricted to $S$ by $X_S$. Similarly subvector of a vector $v$, whose coordinates are restricted to $S$, will be called $v_S$. Throughout the notation $C$ will denote a generic constant unless otherwise specified. The Kronecker delta function will be denoted as  
$\delta_{ij}=\mathds{1}_{i=j}$. For $\mu$ in $\mathbb{R}^m$ ($m>0$) we will also use $\delta_{\mu}$ to define the delta measure which is $1$ at $\mu$ and $0$ otherwise. Since the same notation is used they will appear in separate context and the one referred to will be clear from the context. For $(\mu_1,\ldots,\mu_p)$ with $\mu_i\in \mathbb{R}^m$ ($m,p>0, 1\leq i\leq p$) we will define the empirical distribution by $\frac{1}{p}\sum_{i=1}\delta_{\mu_i}$. For
any sequences of probability measures $\mu_n$ and another probability measure $\mu$, we say that $\mu_n\xRightarrow{W_2}\mu$
if the following holds: there exists a sequence of couplings $\Pi_n$ with marginals $\mu_n$
and $\mu$ respectively, so that if $(X_n,X)\sim \Pi_n$, then $E[(X_n-X)^2]\rightarrow 0$ as $n\rightarrow \infty$.

Next we fix some notations for the optimizers and the optimum value of the risk that will help us state our results in a compact form. Let $\varphi=\beta-\beta_0$ and $\hat{\varphi}_{\mathcal{X}}=\hat{\beta}_{\mathcal{X}}-\beta_0$ where $\hat{\beta}_{\mathcal{X}}$ is the minimizer of the empirical risk in equation \eqref{emp_risk_min}. Consider the expression of empirical risk in equation \eqref{emp_risk_min}. For the sake of brevity, we will denote $\hat{R}_n(Y,\mathcal{X},\beta)$ as $R(\varphi,\mathcal{X})$, where the dependence on $\beta_0$ and $\xi$ (and hence $Y$) is suppressed. Similarly, we will denote the minimum risk as $R(\mathcal{X})$. Further we will denote the general form as $R^K(\varphi,\mathcal{X})$ where $K=L$ or $K=R$ is used to distinguish between the Lasso and ridge setups 
 i.e whether $f(\beta)=\|\beta\|_1/\sqrt{n}$ or $f(\beta)=\|\beta\|^2_2/2$ respectively. The minimum value of the empirical risk in this setup will be denoted as $R^K(\mathcal{X})$ and the error by $\hat{\varphi}^K_{\mathcal{X}}:=\hat{\beta}^K_{\mathcal{X}}-\beta_0$.  Finally we will abuse notation and call both $\hat{\varphi}_{\mathcal{X}}$ and $\hat{\beta}_{\mathcal{X}}$ optimizers when there is no chance of confusion and the quantity referred to will be understood from the context.\\
\section{Examples of block dependent designs\label{sec_examples}}
As discussed in the preceding section, we seek to study models with block dependencies among the design entries. To be precise, we define our formal notion of block dependence below.

\begin{definition}{\label{def_loc_dep}}
A $p$-dimensional random vector $W$ is called block dependent (with dependence parameter $d$) if there exist a partition $\{S_{j}\}_{j=1}^{k_p}$ of $[p]$ (i.e. $S_j$ are disjoint and $\cup_{j=1}^{k_p}S_j=[p]$) such that $|S_{j}|\leq d \,\, \forall j$ and $W_{l}\perp W_k$ if $l\in S_i, k\in S_j,i\neq j$.
\end{definition}

The design matrices whose rows are i.i.d.~block dependent random vectors will be called block dependent designs. As a warm up toward our results, we present here three concrete settings where block dependent designs arise naturally.

\subsection{Penalized additive 
 regression model \label{add_ex}} Let $(Y_i,X_i)_{i=1}^n$ be i.i.d.~with $Y_i\in \mathbb{R}$ and $X_i\in [0,1]^{p_0}$. We consider the following additive model
$$Y_i=h^*(X_i)+\epsilon_i,$$
with $h^*(X_i)=\sum_{j=1}^{p_0}h_j^*(X^j_i)$ and $X^j_i$ denoting the $j^{th}$ component of $X_i$. This model features in \cite{stone_1985}, where nonparametric estimation of $h^*$ is considered under an $L^2$ loss with suitable smoothness restrictions on the function classes. A high-dimensional version of the same problem can be found in \cite{Tan_Zhang}, where the authors consider penalized estimators of the form $\hat{h}=\sum_{i=1}^{p_0}\hat{h}_j$ that minimizes
\begin{equation}\frac{1}{2n}\sum_{i=1}^n(Y_i-h(X_i))^2+\mathcal{D}(h).\label{ell_2_add_reg}\end{equation}
Here $\mathcal{D}(h)$ penalizes the complexity of the function $h$ and the $h_j$'s are the components of the additive model. In this context, \cite{Tan_Zhang} consider a more general model where each component $h_j$ is allowed to be multivariate. For convenience, we first describe the model under a univariate setup. In \cite{Tan_Zhang}, the authors consider an ultra-high-dimensional regime where $\log(p_0)=o(n)$ but with sparsity assumptions on the number of active components $h_j$ (see Section \ref{related_lit} for a detailed discussion). We switch gears to a different regime where ${p_0}/n\rightarrow \kappa$ where $\kappa$ can be greater than 1. Further we do not assume any sparsity on the underlying model and obtain an exact characterization of the estimation risk. In other words, our work encompasses dense as well as sparse signals. In contrast to prior work where the assumed sparsity  is often sub-linear, we are able to allow linear sparse models (see \cite{Tan_Zhang} for other forms of sparsity, such as weak sparsity, that takes a slightly different form, and we refrain from comparing it further here).

To make the setting more concrete, we express each component $h_j$ in terms of the trigonometric basis of $L^2[0,1]$, denoted by
 $\phi_j(t)$, and given by
$$
    \phi_0(t) =1, \quad 
    \phi_{2k-1}(t)=\sqrt{2}cos(2\pi kt), \quad
    \phi_{2k}(t)=\sqrt{2}sin(2\pi kt), \quad k=1,2,\ldots.
$$

In nonparametric estimation theory (see \cite{tsyb_nonparam}) one considers estimation when the functions are restricted to certain smoothness classes with a popular choice being the periodic Sobolev class which is defined as follows
\begin{equation}W(\beta,Q)=\{f \in L^2[0,1]:\theta=\{\theta_l\}\in \Theta(\beta,Q) \text{, where } \theta_l=\int_0^1\phi_l(t)f(t)dt\},\label{sob_2}\end{equation}
where $\Theta(\beta,Q)$ is the Sobolev ellipsoid defined by

\begin{equation}\Theta(\beta,Q)=\{\theta=\{\theta_l \}\in \ell^2(\mathbb{N}):\sum_{l=1}^{\infty}\alpha_l\theta_l^2\leq Q\},\label{Sob_1}\end{equation}
with $\alpha_l=l^{2\beta}$ or $(l-1)^{2\beta}$ for even and odd $l$ respectively.

Let $h_j$ be a function of the form \begin{equation}h_j(t)=\sum_{k=1}^{\infty}\theta^j_k\phi_{k}(t)\label{general_L^2}.
\end{equation}
Then observing that $\theta^j_k$ are Fourier coefficients one can consider estimation of $h_j$ when $h_j\in W(\beta,Q)$ or equivalently 
$\{\theta^j_k\}_{k=1}^{\infty}\in \Theta(\beta,Q)$. Thus the estimation of the coefficients $\theta^j_k$ is equivalent to the estimation of the function $h_j$ and the number of non-zero coefficients is a measure of complexity  of the function. 

In this paper we additionally  impose the condition that $h_j$  is a finite linear combination of the basis functions i.e.  it is a function of the form 
\begin{equation}
    h_j(t)=\sum_{k=1}^{d}\theta^j_k\phi_{k}(t).\label{finite_L^2}
\end{equation}
Indeed such finite linear combinations can be interpreted as a subclass within the periodic Sobolev class with finitely many Fourier coefficients and has been widely studied, most notably in  \cite{BRT}. Using \eqref{ell_2_add_reg} and \eqref{finite_L^2},
we arrive at the penalized estimation problem
$$\frac{1}{2n}\sum_{i=1}^n(Y_i-\sum_{j=1}^{p_0}\sum_{k=1}^d\theta^j_{k}\phi_k(X_i^j))^2+\lambda D(\theta),$$
 where $D(\theta)$ penalizes the complexity of the function $h$ in terms of the parameter $\theta$.
 
 We observe the following facts. Since one can always center $Y$, we may drop the intercept term arising from $\phi_0(t)$. Further, if we assume $X^j_i\sim \text{ Unif }(0,1)$, then using the fact that $\{\phi_k\}$ is an orthonormal basis, we have  $E[\phi_k(X^j_i)]=0$ and $E[\phi_k(X^j_i)\phi_l(X^j_i)]=\delta_{kl}$. Furthermore we assume $X^j_i$s to be independent across $j$. Thus we can reformulate the above problem to estimation of $\beta$ in the following linear regression problem 
\begin{equation}Y=\mathcal{X}\beta+\xi,\label{iso_subG}\end{equation}
where the rows of $\mathcal{X}$ (a $n\times p_0d$ matrix) are of the form
$$\mathcal{X}_i=(\phi_1(X^1_i),\ldots,
    \phi_d(X^1_i),\ldots,\phi_d(X^{p_0}_i),\ldots,
    \phi_d(X^{p_0}_i))^T$$
and $\beta=(\theta^j_k)_{j=1,k=1}^{p_0,d}$. Thus the rows are independent isotropic sub-Gaussian vectors with a block dependent structure with each block having size $d$. We note that letting $p=p_0d$ we may capture properties of estimators in this model by developing theory under the setup where $p_0d/n\rightarrow \kappa$ (see Section \ref{sec_assump}).

Further note that for any estimator $\hat{h}_j(x)$ of the form $\sum_{k=1}^d\hat{\theta}^j_{k}\phi_k(x^{j})$, the $L_2$ estimation risk translates to 
$$\|\hat{h}_j-h_j\|^2:=\int_{0}^1(\hat{h}_j(t)-h_j(t))^2dt=\sum_{k=1}^d|\hat{\theta}^j_{k}-\theta^j_{k}|^2.$$
Thus we have
\begin{equation}\sum_{j=1}^p\|h_j-\hat{h}_j\|^2=\sum_{j=1}^{p_0}\sum_{k=1}^d|\hat{\theta}^j_{k}-\theta^j_{k}|^2\label{add_mod_risk}\end{equation}
and 
it suffices to analyze the estimation risk $\|\hat{\beta}-\beta\|^2_2$ in the model \eqref{iso_subG}.

\subsection{Penalized functional regression \label{fr_ex}}

Let $(Y_i,X_i(t),\epsilon_i)_{i=1}^n$ be i.i.d.~with $Y_i\in \mathbb{R}$ and $X_i(t)\in \mathbb{R}$ for $t\in (0,1)$. We consider the following functional regression model 
\begin{equation}Y_i=\int_{0}^{1}X_i(t)\bgam(t)+\epsilon_i,\label{func_reg_model}\end{equation}
with $\gamma(t)\in \mathbb{R}$ and $\epsilon_i$ being subGaussian random variables.

This model and its variants, called functional linear regression or scalar-on-function regression, has been widely studied in functional data analysis  \cite{func_reg1,func_reg2,kolt_func}.
The Kosambi–Karhunen–Loève theorem states that under standard assumptions the stochastic process $X_i(t)$ admits an expansion with respect to an orthonormal basis i.e
$$X_i(t)=\sum_{k=1}^{\infty}\zeta_{ik}\phi_k(t),$$
where the random variables $\zeta_{ik}$s have mean $0$ and $E[\zeta_{il}\zeta_{ik}]=\lambda_k\delta_{kl}$ i.e. they are uncorrelated with variance $\lambda_k$.
When $X_i(t)$ is expanded upto a finite number of terms as is used in practice  \cite{func_reg1}, 
the regression model can be written as
$$Y_i=\sum_{k=1}^{d}\zeta_{ik}\theta_{k}+\epsilon_i,$$
where we have also expanded the function $\gamma(t)$ as $\bgam(t)=\sum_{k=1}^d\theta_k\phi_k(t)$. In the multivariate setup the data is of the form $(Y_i,\mathbf{X}_i(t),\epsilon_i)_{i=1}^n$ where $Y_i\in \mathbb{R}$ and $\mathbf{X}_i(t)\in \mathbb{R}^p$ for $t\in (0,1).$
We denote the coordinates of the multivariate function $\mathbf{X}_i(t)$ as $X^{j}_i(t)$ and we assume the coordinate functions to be independent. We have the following expansion

$$X_i^{j}(t)=\sum_{k=1}^{d}\zeta^{j}_{ik}\phi_k(t).$$

The random variables $\zeta^{j}_{ik}$ satisfies the relations $E[\zeta^{j}_{ik}]=0$ and $E[\zeta^{j}_{ik}\zeta^{j}_{il}]=\lambda^j_k\delta_{kl}$. We consider the multivariate analogue of the model \eqref{func_reg_model}
$$Y_i=\sum_{j=1}^{p_0}\int_{0}^{1}X^j_i(t)\bgam^j(t)+\epsilon_i.
$$
Finally letting  $\gamma^j(t)=\sum_{k=1}^d\theta^j_k\phi_k(t)$ we obtain  
\begin{equation}
Y_i=\sum_{j=1}^{p_0}\sum_{k=1}^{d}\zeta^j_{ik}\theta^j_{k}+\epsilon_i.\label{fr_expansion}  
\end{equation}
We note that since $X^1_i(t),\ldots X^{p_0}_i(t)$ are independent stochastic processes, \eqref{fr_expansion} can also be written in the form $Y=\mathcal{X}\beta+\xi$ with $\mathcal{X}_i$ of the form
$$\mathcal{X}_i=(\zeta^1_{i1},\ldots,\zeta^1_{id},\ldots,\zeta^{p_0}_{i1},\ldots,\zeta^{p_0}_{id})^T$$
 and $\beta=(\theta^j_k)_{j=1,k=1}^{p_0,d}$. The design matrix $\mathcal{X}$ has independent rows whose entries have the block dependent structure as given in Definition \ref{def_loc_dep}. Indeed for each $i$, $X_i^j(t)$ and $X_i^{m}(t)$ being independent for $m\neq j$, $\zeta^j_{ik}$ and $\zeta^m_{il}$ are independent as well. Further we have $\mathcal{X}=X\Lambda^{1/2}$ where $X_{ij}$s have variance $1$. 
We will assume that that the vector $(\zeta^j_{ik})^d_{k=1}$ is subGaussian random vector \eqref{fr_expansion} and that $E[|\zeta^j_{ik}|^4]\leq K$ for some constant $K$. In particular if $\zeta^j_{ik}$ are bounded and $d$ remains constant the above assumptions trivially holds.
 Also, note that the matrix $\Lambda$ is a diagonal matrix and we will assume that the diagonal entries $\lambda_i$ satisfy
$c\leq \lambda_i\leq C$ for some $c>0$ for all $i \in \{1,\ldots,p_0d\}$.

A common approach towards estimating the vector $\beta$ is to consider the following penalized estimation problem (see \cite{kolt_func} for example)

$$\hat{\beta}=\argmin_{\beta}\frac{1}{2n}\|Y-\mathcal{X}\beta\|^2+f(\beta)$$
and then construct $\hat{\bgam}^{j}(t)=\sum_{k=1}^{d}\hat{\theta}^{j}_{k}\phi_k(t)$. As in the case of Example \ref{ell_2_add_reg} we have
\begin{equation}\sum_{j=1}^{p_0}\|\bgam_j-\hat{\bgam}_j\|^2=\sum_{j=1}^{p_0}\sum_{k=1}^d|\hat{\theta}^j_{k}-\theta^j_{k}|^2\label{func_mod_risk},\end{equation}
and we are reduced to the familiar setup of finding an asymptotic expression for  $\|\hat{\beta}-\beta\|^2_2$ in the model \eqref{iso_subG}.

\begin{remark} We note that while the concept that nonparametric models can be studied as linear models via basis expansion is widely studied in the literature, the structure of dependence arising from such models is the focus of our study. Unless Gaussian assumptions are used (for example assuming $X(t)$ is a Gaussian process) the design matrix contains rows with dependent entries. In the first and second model, we may assume that $d$ is finite and $n/(p_0d)\rightarrow \kappa^{-1}$. This assumption can be relaxed to the assumption $c<n/p_0<C$ for some constants $c$ and $C$. In the more general setup we can allow $d$ to vary i.e
$k \in \{1,\ldots d_j\}$ and $cn<\sum_{j=1}^pd_j<Cn$. We will assume that each $d_i$ is bounded but in the most general setup they can grow polylogarithmically with $n$.\end{remark}

 \subsection{Block dependent covariates in GWAS \label{GWAS_ex}}
In Genome Wide Association Studies (GWAS) one often studies the association between phenotypes ($Y_i$'s) and genetic markers/single nucleotide polymorphism (SNP) $X_{ij}$ where $i$ denotes the $i^{th}$ individual and $j\in \{1,\ldots,p\}$. The markers $X_{ij}$ typically follow a dependence structure where a particular SNP is dependent on SNPs of nearby loci, a phenomenon called \textit{linkage disequilibrium}. This phenomenon is taken into account in the detection of causal SNPs and we refer the readers to \cite{knkf_1} and the references therein for a detailed account. 

While most measures of linkage disequilibrium is based on correlation (see \cite{Lewontin}), other measures of association has been considered as well. For example, \cite{LD_gen} describes a measure of association  in terms of Kullback-Leibler distance between distributions. The same reference also points out that situations of dependence may arise where the associated variables may be uncorrelated despite being dependent and that correlation is rarely useful beyond the case of Gaussian designs. Furthermore in GWAS, the markers $X_{ij}$ typically take on three different discrete values, encoding allele frequencies, and therefore this application rarely encounters  Gaussian designs. These observations are perfectly in tandem with the general model that we consider in this paper. 
To elucidate the occurrence of uncorrelated but dependent random variables with discrete support we consider the following toy example:\\

\begin{center}
    \begin{tabular}{cccc}
         X& A& B&C\\
         $a_1$& 1& 0&0\\
         $a_2$& -1& 0&0\\
         $a_3$& 0& 1&0\\
         $a_4$& 0& -1&0\\
         $a_5$& 0& 0&1\\
         $a_6$& 0& 0&-1.\\
    \end{tabular}
\end{center}
Let $X$ be a random variable taking six distinct values $a_1,\ldots, a_6$ with equal probabilities. Further let the three random variables take values according to the entries of the $i^{th}$ row i.e. when $X=a_1$ we have $A=1,B=0,C=0$, and so forth. It can be easily verified that the marginal distribution of $A,B,C$ are same and they have mean $0$. Further their covariance matrix is $cI_3$ for some constant $c$, so that  the vector $(A,B,C)$ can be easily scaled to obtain an isotropic random vector. Further because of the underlying generation mechanism, they are dependent. We note that this can be easily generalized to $d$ random variables from a table with $2d$ rows. 

On the other hand, the other aspect of the dependence structure arising in GWAS is the concept of neighborhood dependence that tells us that a particular SNP depends only on its neighbors. Of particular importance is the block structure considered in \cite{LD_block}, where there is linkage disequilibrium if and only if two SNPs belong to the same block. Our block dependence structure in this paper thus encompasses this critical GWAS example. 

After discussing our motivation and relevant examples, we now proceed to present the preliminaries and assumptions necessary for our main results.

\section{Preliminaries\label{prelim}}
\subsection{Assumptions \label{sec_assump}}
Throughout we assume that we observe n i.i.d.~samples $\{Y_i,\mathcal{X}_i,1\leq i\leq n\}$.
We will work in a high-dimensional setup where the covariate dimension $p$
is allowed to grow with the sample size $n$. To make this rigorous, we consider a sequence
of problem instances $\{Y_{i,n},\mathcal{X}_{i,n},\xi_{i,n},1\leq i\leq n,\beta_0(n),\beta(n)\}_{n\geq 1}$ and work under the setting of a linear model
$$Y_{i,n}=\mathcal{X}^T_{i,n}\beta_0(n)+\xi_{i,n}, $$
with $Y_{i,n}\in \mathbb{R},\mathcal{X}_{i,n}\in \mathbb{R}^{p(n)}$, and $\beta_0(n)\in \mathbb{R}^{p(n)}$. We allow $p(n),n\rightarrow \infty$ with $p(n)/n\rightarrow \kappa \in (0,\infty)$ and in the sequel we will drop the dependence on $n$ whenever it is clear from the context. Further we let 
$\xi_i$s be independent uniformly sub-Gaussian random variables with variance $\sigma^2$. Note they do not need to arise from the same distribution. .
In a more compact form the above model will be written as
\begin{equation}Y=\mathcal{X}\beta_0+\xi,\label{lin_mod}\end{equation}  
where $\mathcal{X}\in\mathbb{R}^{n \times p}, Y,\xi\in \mathbb{R}^n$. 

Our assumptions on the design $\mathcal{X}$, the signal $\beta_0$, and the tuning parameter $\lambda$ are as follows. 

\begin{assumption}(Assumptions on the design)\label{ass_1} The rows of the design $\mathcal{X}$ are i.i.d.~random vectors each of which is block dependent with dependence parameter $d$. We assume that $\mathcal{X}=X\Lambda^{1/2}$ where the rows of $X$ are centered and isotropic i.e. if $X_i$ denotes the $i$-th row of $X$ then $E[X_i]=0, E[X_{ij}X_{ik}]=\delta_{jk}$. Further we assume that the entries have bounded fourth moment ($E|X_{ij}|^4\leq K$) and the row vector $X_i$ is a subGaussian random vector. Furthermore, we assume that $\Lambda$ is a diagonal matrix with the diagonal entries $\lambda_i$ satisfying $c\leq \lambda_i\leq C$ for some positive constants $c,C$.
\end{assumption}

As mentioned in the remark at the end of Subsection \ref{fr_ex}, the dependence occurs in blocks and the block sizes should be bounded by the dependence parameter $d$.
In  Section \ref{sec_prf_otln}, we establish that our universality holds as long as $d=o\left(\left(\frac{n}{\log n}\right)^{1/5}\right)$.
Finally, the assumption $E[X_{ij}^4]\leq K$ can be weakened to an assumption involving $d,n$ and the moments, but for simplicity, we will state our theorems under this simplified assumption.

\begin{assumption}{\label{ass_2}}
The signal satisfies the bound $\beta^T_0\Lambda \beta_0\leq C$. Let $v_i=\sqrt{n\lambda_i}\beta_{0,i}$ for $1\leq i\leq p$. Then we assume that
$$\frac{1}{p}\sum_{i=1}^p\delta_{\lambda_i,v_i}\xRightarrow{W_2}\mu,$$
for some measure $\mu$ supported on $\mathbb{R}^{\geq 0}\times\mathbb{R}$. Further we assume that $\|v\|_{\infty}\leq C$.\end{assumption}

\begin{assumption}{\label{ass_3}}
The tuning parameter $\lambda$ satisfies $\lambda\geq \lambda_0$ for a constant $\lambda_0$.
\end{assumption}

We remark that we require this assumption only for the case of the Lasso and is the analog of restricted strong convexity type assumptions in the Lasso literature (see Section \ref{opt_bound} for further details).

\subsection{Choice of penalty function}

 As will be shown in Section \ref{sec_prf_otln}, universality of both the empirical risk and the estimation risk depends on establishing apriori that the optimizers are guaranteed to lie in ``nice'' sets. However establishing this property requires a case-by-case analysis. 
 In this paper, we consider the two most popular choices of penalty used in the statistics literature---the Lasso and the ridge---that can form prototypes for many other popular examples, the ridge serving as a prototype for strong convex penalties and the Lasso serving as a prototype in the absence of strong convexity. Since we are able to handle these prototypes, we believe our method can be extended to handle other separable penalties.

\begin{remark}Some comments on the scaling of the penalties are also in order. 
In \cite{han2023universality}) 
the entries of $\mathcal{X}$ are $O_p(n^{-1/2})$ and $f(\beta)=\|\beta\|_1/n$ or $f(\beta)=\|\beta\|^2_2/2n$ respectively. However in our case, the entries of $\mathcal{X}$ are $O_p(1)$. Writing $A=\mathcal{X}/\sqrt{n}$ and $\mu=\sqrt{n}\beta$, equation \eqref{emp_risk_min} can be rewritten in the form
\begin{equation}\hat{R}_n(y,A,\mu)=\frac{1}{2n}\sum_{i=1}^n\|Y_i-A^T_i\mu\|^2+\lambda \bar{f}(\mu),
\label{emp_risk_min_2}\end{equation}
where $\bar{f}(\mu)=\|\mu\|_1/n$ or $\bar{f}(\mu)=\|\mu\|^2_2/2n$ and we recover the setup in \cite{han2023universality}. We have avoided the $n^{-1/2}$ scaling in the design matrix since it helps to state the assumptions in a concise fashion and to borrow certain results from the random matrix literature. In our setup we end up with the non-standard scaling of the penalty i.e. $f(\beta)=\|\beta\|_1/\sqrt{n}$ or $f(\beta)=\|\beta\|^2_2/2$ respectively.\end{remark}

\subsection{Fixed point equations for the estimation risk}
Finally, our universality result states that the estimation risk of regularized estimators under our setting is asymptotically the same as that in a setting where the design has i.i.d.~Gaussian entries. Extensive prior work characterizes the risk of convex regularized estimators in the Gaussian case \cite{bayati_mont,elkaroui_pnas_1,donoho2016high}. In this characterization, a certain system of fixed point equations plays a critical role. We review this system below for convenience. Let
$$\eta_{k}(x,\lambda):=\argmin_{x\in \mathbb{R}^p}\Big\{\frac{1}{2}\|z-x\|^2+\lambda\frac{\|x\|^k_k}{k}\Big\}.$$
We will focus on two particular values of $k$ in the above expression of $\eta_k(,)$; $k$ be taken as $2$ when $K=R$ and $1$ when $K=L$ (i.e. ridge and Lasso respectively).

Now let $\mu_{0,i}=\sqrt{n\lambda_i}\beta_{0,i}$ and $\omega_i=\lambda_i^{-1/2}$. Also define $Q_n$ as the law of the $3$-tuple of random variables $(M,\Omega,Z)\sim (\frac{1}{p}\sum_{i=1}^p\delta_{\mu_{0,i},\omega_i})\otimes N(0,1)$. We have the following system of equations
\begin{align}
    (\gamma^K_*)^2&=1+\frac{p}{n}E_{Q_n}\left[\eta_k\left(M+\gamma^K_*Z,\frac{\gamma^K_*\lambda\Omega^k}{\beta^K_*}\right)-M\right]^2,\label{eqn_system}\\
    \beta^K_*&= \gamma^K_*\left[1-\frac{p}{n}E_{Q_n}\left[\eta'_k\left(M+\gamma^K_*Z,\frac{\gamma^K_*\lambda\Omega^k}{\beta^R_*}\right]\right)\right]\nonumber.
    \end{align}

Consistent with our notation, $(\beta_*^R,\gamma_*^R)$ will be the solution of the fixed point equations \eqref{eqn_system} for ridge and $(\beta_*^L,\gamma_*^L)$ 
 will the corresponding solution for the Lasso. Next we define ``error" terms which are functions of the solutions of the above equation system. Let $g\sim N(0,I_p)$. Define the vectors $\varphi^R_{*}$ and $\varphi^L_{*}$ for the ridge and Lasso problems in terms of their $i^{th}$ entries as follows
\begin{align}\varphi^R_{*,i}&=\eta_2\left(\beta_{0,i}+\frac{\gamma^R_*\omega_ig_i}{\sqrt{n}},\frac{\gamma^R_*\lambda\omega_i^2}{\beta^R_*}\right)-\beta_{0,i},\label{ridge_limit_phi}\\
\varphi^L_{*,i}&=\eta_1\left(\beta_{0,i}+\frac{\gamma^L_*\omega_ig_i}{\sqrt{n}},\frac{\gamma^L_*\lambda\omega_i^2}{\sqrt{n}\beta^L_*}\right)-\beta_{0,i},\label{lasso_limit_phi}
\end{align}
for $i \in \{1,\ldots,p\}$.
The above term will play a crucial role in the asymptotic characterization of the estimation risk.
We are now ready to state the main results of our paper.

\section{Main result \label{main_res}}
In this section we state the main results of our paper. We consider the linear model \eqref{lin_mod} and seek to study universality properties of the optimum value of the empirical risk and the estimation risk of the  optimizer. In particular, we aim to establish that 
\begin{align}\label{eq:heuristic}
\min_{\beta}\hat{R}_n(Y,\mathcal{X},\beta)\approx \min_{\beta}\hat{R}_n(Y,\mathcal{\mathcal{G}},\beta) \quad \text{and} \quad 
    \|\hat{\beta}_{\mathcal{X}}-\beta_{0}\|^2  \approx \|\hat{\beta}_{\mathcal{G}}-\beta_{0}\|^2,
\end{align}
where $\mathcal{G}=G\Lambda^{1/2}$, $G$ a matrix with i.i.d.~standard Gaussian entries. 

Recall that we have defined the shifted versions of the optimizers as $\hat{\varphi}^K_{\mathcal{X}}:=\hat{\beta}^K_{\mathcal{X}}-\beta_0$. Before we state our main theorems, we describe a set restricted to which our universality result holds. $$\mathcal{T}_n:=\{\varphi:\|\varphi\|_{\infty}\leq C\sqrt{\frac{d\log n}{n}},\|\varphi\|\leq C\}.$$

\begin{theorem}{\label{thm_optimum}}
   Suppose the design matrix $\mathcal{X}=X\Lambda^{1/2}$ satisfies Assumption \ref{ass_1} and $\beta_0$ satisfies Assumption 
 \ref{ass_2}. Further $\mathcal{G}=G\Lambda^{1/2}$, where $G \in \mathbb{R}^{n \times p}$ contains i.i.d.~Gaussian entries, and that $\psi$ is a bounded differentiable function with bounded Lipschitz derivative. Then $\exists \gamma_n=o(1)$ such that for any subset $S_n\subset \mathbb{R}^p$, we have\\
   $$|E(\psi(\min_{\varphi\in \mathcal{S}_n\cap \mathcal{T}_n}R^{R}(\varphi,\mathcal{X})))-E(\psi(\min_{\varphi\in \mathcal{S}_n\cap \mathcal{T}_n}R^R(\varphi,\mathcal{G})))|\leq \gamma_n.$$
   Further if we assume Assumption \ref{ass_3} in addition to Assumptions \ref{ass_1}-\ref{ass_2}, then we have
   $$|E(\psi(\min_{\varphi\in \mathcal{S}_n\cap \mathcal{T}_n}R^{L}(\varphi,\mathcal{X})))-E(\psi(\min_{\varphi\in \mathcal{S}_n\cap \mathcal{T}_n}R^L(\varphi,\mathcal{G})))|\leq \gamma_n.$$
\end{theorem}
 In the next step we show that the Lasso and ridge optimizers denoted by  $\hat{\varphi}^L_{\mathcal{X}}$ and $\hat{\varphi}^R_\mathcal{X}$ lie in the subset $\mathcal{T}_n$ with high probability. Thus the restriction over $\mathcal{T}_n$ can be easily lifted and we have the following result.

\begin{corollary}{\label{opt_risk}}
    Assume the conditions of Theorem \ref{thm_optimum}. Then we have 
    $$|E(\psi(\min_{\varphi} R^{K}(\varphi,\mathcal{X})))-E(\psi(\min_{\varphi}R^K(\varphi,\mathcal{G})))|\leq \gamma_n,$$
    for some $\gamma_n=o(1)$. In particular if $R^K(\mathcal{G})\rightarrow \rho$, then $R^K(\mathcal{X})\rightarrow \rho$ as well.
\end{corollary}

Next we turn to characterizing the estimation risk of the optimizers. The difficulty and the methods used to prove this result is thoroughly discussed in Section \ref{sec_prf_otln}. Finally, we have our main result.

\begin{theorem}{\label{thm_optimizer}}
    Assume the conditions in Theorem \ref{thm_optimum}. For all Lipshitz functions $\varrho:\mathbb{R}^p\rightarrow \mathbb{R}$ and $\epsilon>0$ we have 
   \begin{equation}\label{eq:univoptim}
   P(|\varrho(\hat{\varphi}^K_X)-\varrho(\hat{\varphi}^K_G)|\geq \epsilon)\leq \varepsilon_n,
   \end{equation}
for some $\varepsilon_n\rightarrow 0$.
\end{theorem}

We will establish in the proof of Theorem \ref{thm_optimizer} that both $\varrho(\hat{\varphi}^K_X)$ and $\varrho(\hat{\varphi}^K_G)$ concentrate around the common value $E[\varrho(\varphi^K_*)]$ where $\varphi^R_*$ and $\varphi^L_*$
are defined in equations \eqref{ridge_limit_phi} and \eqref{lasso_limit_phi} respectively. Thus by choosing a suitable $\varrho$, the above theorem can naturally be specialized to yield the estimation risk of the ridge and the Lasso in our dependent covariates setup. We obtain the following corollary for dependent subGaussian designs. 

\begin{corollary}{\label{est_err}}
Let $\gamma_*^K$ and $\beta_*^K$ be fixed points of the equation system \eqref{eqn_system}. In the setup of Theorem \ref{thm_optimum}, we have   
\begin{align*}\|\hat{\beta}^R_{\mathcal{X}}-\beta_{0}\|^2-E_{\bar{Q}_n}\left[\eta_2\left(M+\frac{\gamma^R_*\Omega Z}{\sqrt{n}},\frac{\gamma^R_*\lambda\Omega^2}{\beta^R_*}\right)-M\right]^2&\xrightarrow{P}0,\\
\|\hat{\beta}^L_{\mathcal{X}}-\beta_{0}\|^2-E_{\bar{Q}_n}\left[\eta_1\left(M+\frac{\gamma^L_*\Omega Z}{\sqrt{n}},\frac{\gamma^L_*\lambda\Omega^2}{\sqrt{n}\beta^L_*}\right)-M\right]^2&\xrightarrow{P}0.
\end{align*}
Here the expectations are taken with respect to the measure  $\bar{Q}_n$ which is the law of the $3$-tuple of random variables $(M,\Omega,Z)\sim (\frac{1}{p}\sum_{i=1}^p\delta_{\beta_{0,i},\omega_i})\otimes N(0,1)$ where $\omega_i=\lambda^{-1/2}_i$. 
\end{corollary}
Note that the equivalent results for estimation risk in the  independent Gaussian design have been obtained in \cite{elkaroui_ptrf} and \cite{bayati_mont} for ridge and Lasso respectively and thus our result generalizes the exact risk characterization to the dependent subGaussian case. We are now in position to compute the estimation risks of the Lasso and the ridge in the penalized regression models described in Section \ref{sec_examples}. In the following corollaries, we demonstrate these for the two nonparametric regression models we described. 

\begin{corollary}{\label{addreg_corr}}
Let $$Y_i=\sum_{j=1}^{p_0}h^*_j(X^j_i)+\epsilon_i, \quad  i \in \{1,\ldots,n\},$$ where $X^j_i$ are i.i.d.~Unif(0,1), and $h^*_j(t)=\sum_{k=1}^d\theta^{j}_{0k}\phi_k(t)$. Define $\theta=\{\theta_k^{j}\}_{j=1,k=1}^{p_0,d}$ and $\beta_0=\{\theta_{0,k}^j\}_{j=1,k=1}^{p_0,d}$ and let $\beta_0$ satisfy Assumption \ref{ass_1} with $\Lambda=I_{p_0d}$. Further define the optimizers as follows 
\begin{align*}
    \hat{\theta}^L&=\argmin_{\theta}\frac{1}{2n}\sum_{i=1}^n(Y_i-\sum_{j=1}^{p_0}\sum_{k=1}^d\theta^j_{k}\phi_k(X_i^j))^2+\frac{\lambda}{\sqrt{n}}\sum_{j=1}^{p_0}\sum_{k=1}^d|\theta^j_{k}|,\\
    \hat{\theta}^R&=\argmin_{\theta}\frac{1}{2n}\sum_{i=1}^n(Y_i-\sum_{j=1}^{p_0}\sum_{k=1}^d\theta^j_{k}\phi_k(X_i^j))^2+\frac{\lambda}{2}\sum_{j=1}^{p_0}\sum_{k=1}^d|\theta^j_{k}|^2.
\end{align*}
Then denoting $\hat{h}^L_j(t)=\sum_{k=1}^d\hat{\theta}_k^{j,L}\phi_k(t)$ and $\hat{h}^R_j(t)=\sum_{k=1}^d\hat{\theta}_k^{j,R}\phi_k(t)$ we have,
\begin{equation}\sum_{j=1}^{p_0}\|\hat{h}^R_j-h^*_j\|^2-E_{\tilde{Q}_n}\left[\eta_1\left(M+\frac{\gamma^L_*Z}{\sqrt{n}},\frac{\gamma^L_*\lambda}{\sqrt{n}\beta^L_*}\right)-M\right]^2\xrightarrow{P}0.\label{addreg_ridge_risk}\end{equation}
Further if Assumption \ref{ass_3} is satisfied we have
\begin{equation}\sum_{j=1}^{p_0}\|\hat{h}^L_j-h^*_j\|^2-E_{\tilde{Q}_n}\left[\eta_2\left(M+\frac{\gamma^R_*Z}{\sqrt{n}},\frac{\gamma^R_*\lambda}{\beta^R_*}\right)-M\right]^2\xrightarrow{P}0.\label{addreg_lasso_risk}\end{equation}
In the above two statements the expectations are taken with respect to the measure  $\tilde{Q}_n$ which is the law of the $2$-tuple of random variables $(M,Z)\sim (\frac{1}{p}\sum_{i=1}^p\delta_{\beta_{0,i}})\otimes N(0,1)$. 
\end{corollary}

\begin{corollary}{\label{funcreg_corr}}
    Let 
    $$Y_i=\sum_{j=1}^{p_0}\int_{0}^{1}X^j_i(t)\bgam^j_0(t)+\epsilon_i,\quad i \in \{1,\ldots,n\},$$
where $X_i^j(t)$s are independent, $X_i^j(t)=\sum_{k=1}^{d}\zeta^j_{ik}\phi_k(t)$, and $\bgam^j_0(t)=\sum_{i=1}^d\theta^j_{0k}\phi_k(t)$. Define $\theta=\{\theta_k^{j}\}_{j=1,k=1}^{p_0,d}$, $\beta_0=\{\theta_{0,k}^j\}_{j=1,k=1}^{p_0,d}$, $Var(\zeta^j_{ik})=\lambda_{(j-1)d+k}$, and $\Lambda=diag(\lambda_1,\ldots,\lambda_{p_0d})$. Let the vectors $\zeta_i=(\zeta^j_{ik})_{j=1,k=1}^{p_0,d}$ satisfy Assumption \ref{ass_1}. Further let $\beta_0$ satisfy Assumption \ref{ass_2}. Define the optimizers as follows 
\begin{align*}
    \hat{\theta}^L&=\argmin_{\theta}\frac{1}{2n}\sum_{i=1}^n(Y_i-\sum_{j=1}^{p_0}\sum_{k=1}^d\theta^j_{k}\zeta^j_{ik})^2+\frac{\lambda}{\sqrt{n}}\sum_{j=1}^p\sum_{k=1}^d|\theta^j_{k}|,\\
    \hat{\theta}^R&=\argmin_{\theta}\frac{1}{2n}\sum_{i=1}^n(Y_i-\sum_{j=1}^{p_0}\sum_{k=1}^d\theta^j_{k}\zeta^j_{ik})^2+\frac{\lambda}{2}\sum_{j=1}^{p_0}\sum_{k=1}^d|\theta^j_{k}|^2.
\end{align*}
Then denoting $\hat{\bgam}^L_j(t)=\sum_{k=1}^d\hat{\theta}_k^{j,L}\phi_k(t)$ and $\hat{\bgam}^R_j(t)=\sum_{k=1}^d\hat{\theta}_k^{j,R}\phi_k(t)$ we have
\begin{equation}
\sum_{j=1}^{p_0}\|\hat{\bgam}^R_j-\bgam^*_j\|^2-E_{\bar{Q}_n}\left[\eta_2\left(M+\frac{\gamma^R_*\Omega Z}{\sqrt{n}},\frac{\gamma^R_*\lambda\Omega^2}{\beta^R_*}\right)-M\right]^2\xrightarrow{P}0.\label{funcreg_ridge_risk}\end{equation}
Further if Assumption \ref{ass_3} is satisfied we have

\begin{equation}\sum_{j=1}^{p_0}\|\hat{\bgam}^L_j-\bgam^*_j\|^2-E_{\bar{Q}_n}\left[\eta_1\left(M+\frac{\gamma^L_*\Omega Z}{\sqrt{n}},\frac{\gamma^L_*\lambda\Omega^2}{\sqrt{n}\beta^L_*}\right)-M\right]^2\xrightarrow{P}0.\label{funcreg_lasso_risk}\end{equation}
Here the expectations are taken with respect to the measure  $\bar{Q}_n$ which is the law of the $3$-tuple of random variables $(M,\Omega,Z)\sim (\frac{1}{p}\sum_{i=1}^p\delta_{\beta_{0,i},\omega_i})\otimes N(0,1)$ where $\omega_i=\lambda^{-1/2}_i$.
\end{corollary}

\begin{remark}
    Corollaries \ref{addreg_corr} and \ref{funcreg_corr} follows directly by applying Corollary \ref{est_err} to the additive and functional regression models respectively. The simpler form of the risks in \eqref{addreg_ridge_risk} and \eqref{addreg_lasso_risk} in comparison to \eqref{funcreg_ridge_risk} and \eqref{funcreg_lasso_risk} is due to the fact the the rows in the former setup are isotropic random vectors and hence $\Lambda=\tilde{\Omega}=I_{p_0d}$. Similarly, in this simplified setup the measure $\bar{Q}_n$ appearing Corollary \ref{est_err} simplifies to $\tilde{Q}_n$ in \eqref{addreg_ridge_risk} and \eqref{addreg_lasso_risk}.
\end{remark}
\section{Illustration\label{sec_simul}}
In this section, we demonstrate the efficacy of our result via experiments. In particular, we consider the first two settings in Section \ref{sec_examples} and compare the risk of regularized estimators in each model with that under  i.i.d.~Gaussian designs. Across our simulations, we observe that universality kicks in for quite moderate sample sizes, a reassuring observation that suggests our current asymptotic abstraction captures reality reasonably well. As in Section \ref{sec_examples} we will denote $p=p_0d$, i.e.~there will be $p_0$ blocks each having size $d$. 

\subsection{Penalized additive regression model}

Recall the setup given in Example \ref{ell_2_add_reg},
$$Y_i=\sum_{i=1}^{p_0}h_j^*(X^j_i)+\epsilon_i,\quad h_j(t)=\sum_{k=1}^{d}\theta^j_k\phi_{k}(t), \quad \phi_{2l-1}(t)=\sqrt{2}cos(2\pi lt), \quad
    \phi_{2l}(t)=\sqrt{2}sin(2\pi lt)\text{ for } l=1,\ldots,d/2.$$

We set $n=200,p_0=30, d=10$ so that $(p_0d)/n=1.5$. For the $i^{th}$ row of the design matrix, we use the following strategy: Generate i.i.d.~$\{X^j_i\}_{j=1}^{30}$ such that $X^j_i\sim Unif(0,1)$ and construct  the following

$$\mathcal{X}_i=(\phi_1(X^1_i),\ldots,
    \phi_{10}(X^1_i),\ldots,\phi_1(X^{30}_i),\ldots,
    \phi_{10}(X^{30}_i))^T.$$

We repeat this construction for $n=200$ rows to create the design matrix $\mathcal{X}$ and for the Gaussian problem, we construct the $200\times 300$ 
 matrix $\mathcal{G}$ with i.i.d.~Gaussian entries. We generate the error entries $\epsilon_i$, $\bar{\epsilon}_i$ so that they are i.i.d.~standard normal. Finally, we generate the signal $\beta_0=\{\theta^j_{0,k}\}^{p_0,d}_{j=1,k=1}$ as $\beta_{0,j}=n^{-1/2}\mathfrak{b}_j$, where the $\mathfrak{b}_j's$ are i.i.d.~Bernoulli random variables, and subsequently, keep $\beta_0$ fixed across multiple runs/iterations of the simulation.  For each iteration $m$ (we will average over 50 iterations later), we generate data corresponding to our two  models under consideration as follows
\begin{equation}Y_i=\mathcal{X}^T_i\beta_0+\epsilon_i, \quad Y_i^{\mathcal{G}}=\mathcal{G}^T_i\beta_0+\bar{\epsilon}_i.\label{gen_mod}\end{equation}

We recall from \eqref{add_mod_risk} that in this model
$$\sum_{j=1}^p\|h_j-\hat{h}_j\|^2=\|\hat{\beta}^L_{\mathcal{X}}-\beta_0\|^2,$$
if the Lasso is used and $\|\hat{\beta}^R_{\mathcal{X}}-\beta_0\|^2$ in the case of ridge regression.
In Figures \eqref{plotaddL} and \eqref{plotaddR}, we plot the estimation risks $\|\hat{\beta}^L_{\mathcal{X}}-\beta_0\|^2$ and $\|\hat{\beta}^R_{\mathcal{X}}-\beta_0\|^2$ averaged over $50$ iterations with red curves. We overlay the corresponding risks for the Gaussian problem, $\|\hat{\beta}^L_{\mathcal{G}}-\beta_0\|^2$ and $\|\hat{\beta}^R_{\mathcal{G}}-\beta_0\|^2$, as black dots. Observe the impeccable agreement of the risks under both settings that clearly validates our universality results and demonstrates its impressive efficacy in sample sizes as low as a couple of hundred.

\begin{figure}[hbt!] 
\begin{subfigure}{0.5\textwidth}
\centering
\includegraphics[scale=0.425,keepaspectratio]{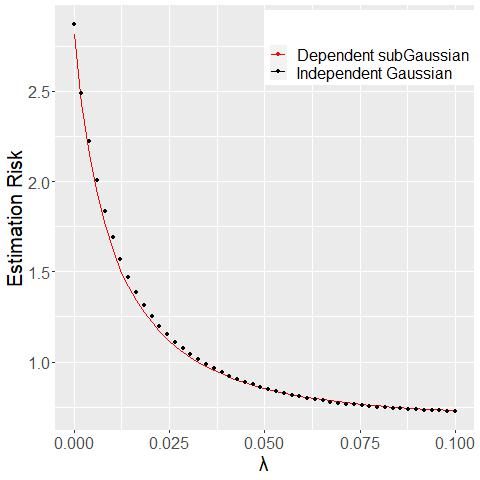} 
\caption{Lasso}
\label{plotaddL}
\end{subfigure}
\begin{subfigure}{0.5\textwidth}
\centering
\includegraphics[scale=0.425,keepaspectratio]{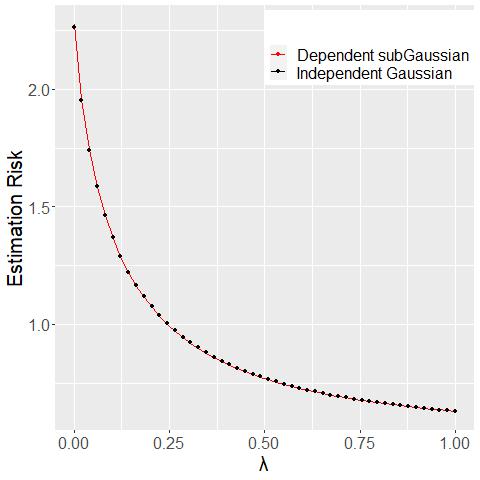}
\caption{Ridge}
\label{plotaddR}
\end{subfigure}
\label{fig:image1}
\caption{Estimation risk in the penalized additive model (see Section \ref{add_ex}) averaged over 50 iterations with $n=200,p_0=30$, and $d=10$. The red curves 
 correspond to the risk $\|\beta^L_{\mathcal{X}}-\beta_0\|^2$ and $\|\beta^R_{\mathcal{X}}-\beta_0\|^2$ for the Lasso and the ridge respectively, while the corresponding risks for the Gaussian problem $\|\beta^L_{\mathcal{G}}-\beta_0\|^2$ and $\|\beta^R_{\mathcal{G}}-\beta_0\|^2$ are overlaid as black dots.}
\end{figure}

\subsection{Penalized functional regression}

Recall that in the setup given in Example \ref{fr_ex} we have,

$$Y_i=\sum_{j=1}^{p_0}\sum_{k=1}^{d}\zeta^{j}_{ik}\beta^{j}_{k}+\epsilon_i,$$
where $E[\zeta^{j}_{ik}]=0$ $E[\zeta^{j}_{ik}\zeta^{j}_{il}]=\delta_{kl}\lambda_k$. Further $\zeta^{j}_{ik}$ and $\zeta^{m}_{il}$ are independent if $j\neq m$. We consider the case when $n=500$ and $p_0=30$, and $d=10$ i.e. we construct a setup where each row has $30$ blocks of size $10$ so that so that $(p_0d)/n=0.6$. The random variables across the block 
 are independent while the random variables within the blocks are dependent. To construct dependent random variables within a block, we generate $10$ independent Rademacher random variables and multiply them by a common random variable, as defined below. We repeat this procedure for every block. Let $V$ be a random variable distributed as follows
 $$P(V=0)=\frac{1}{2},\quad P(V=-\sqrt{2})=\frac{1}{4},\quad P(V=\sqrt{2})=\frac{1}{4}.$$
 We generate i.i.d.~rows $X_i$ with $j$-th entry $X_{ij}$ distributed as follows. Let $\{U_{j}\}_{j=1}^{300}$ and $\{V_{j}\}_{j=1}^{30}$ be two sets of independent random variables with $U_j$'s distributed as i.i.d.~Rademacher and $V_j$s as i.i.d.~random variables with the distribution same as that of $V$. Now define $W_j=U_jV_{\lceil\frac{j}{10}\rceil}$ for $j\in \{1,\ldots,300\}$. Next we generate the diagonal matrix $\Lambda$ with $30$  identical blocks: the diagonal elements of the first block $\lambda^{1/2}_i, i\in 1\ldots 10$  are generated independently from $Unif(1,2)$ and all other blocks are identical copies of the first block. It can be easily verified that the random variables are pairwise uncorrelated and the random variables within the blocks are not independent. 
We generate the design matrix $\mathcal{X}$ as $\mathcal{X}=X\Lambda^{1/2}$. For the Gaussian case, we generate a $500\times 300$ matrix $G$ with i.i.d.~Gaussian entries and then define $\mathcal{G}=G\Lambda^{1/2}$. As in the previous subsection, we generate the signal $\beta_0=\{\theta^j_{0,k}\}^{p_0,d}_{j=1,k=1}$ as $\beta_{0,j}=n^{-1/2}\mathfrak{b}_j$, where the $\mathfrak{b}_j's$ are i.i.d.~Bernoulli random variables, and subsequently, keep $\beta_0$ fixed across multiple runs/iterations of the simulation.
 Finally, we generate the random variables $Y$ and $Y^{\mathcal{G}}$ following the linear models in \eqref{gen_mod}  (after adjusting for the dimension).

In the aforementioned setting, we plot the  estimation risk of the Lasso and the ridge as a function of $\lambda$. Figure \eqref{plotfuncL} shows Lasso with the red curve being the case of the dependent sub-Gaussian design and the black dots being the independent Gaussian design. Figure \eqref{plotfuncR} shows the corresponding plot  for the ridge. In both cases, the risk curves match, demonstrating the validity of our universality result and the fact that the asymptotics kicks in for quite moderate sample size and dimensions.

\begin{figure}[hbt!] 
\begin{subfigure}{0.5\textwidth}
\centering
\includegraphics[scale=0.425,keepaspectratio]{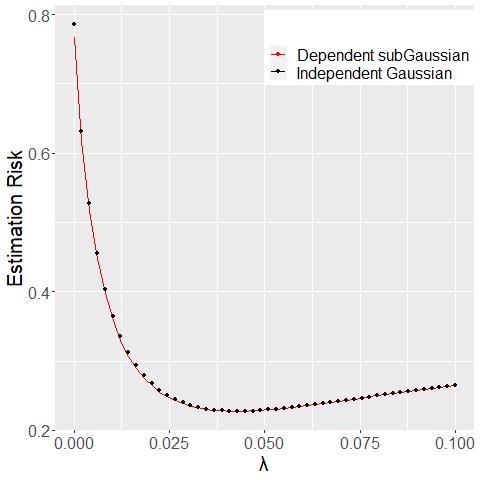} 
\caption{Lasso}
\label{plotfuncL}
\end{subfigure}
\begin{subfigure}{0.5\textwidth}
\centering
\includegraphics[scale=0.425,keepaspectratio]{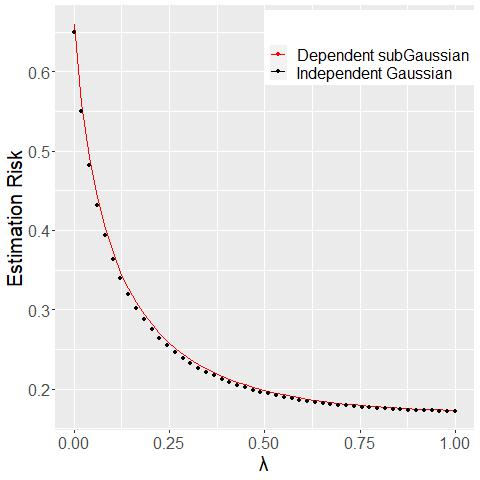}
\caption{Ridge}
\label{plotfuncR}
\end{subfigure}
\label{fig:image2}
\caption{Estimation risk in the penalized functional regression model (see Section \ref{fr_ex}) averaged over 50 iterations with $n=500,p_0=30$, and $d=10$. The red curves corresponds to the risk $\|\beta^L_{\mathcal{X}}-\beta_0\|^2$ and $\|\beta^R_{\mathcal{X}}-\beta_0\|^2$ for Lasso and risk respectively, while the corresponding risks for the Gaussian $\|\beta^L_{\mathcal{G}}-\beta_0\|^2$ and $\|\beta^R_{\mathcal{G}}-\beta_0\|^2$ are overlaid as black dots.}
\end{figure}

\section{Related Literature \label{related_lit}}

In this section, we provide a detailed discussion of existing universality results in the literature, and situate our work in this broader context.
Universality results are pervasive in the literature on random matrix theory. Over the past decade, the exploration of universality concerning the choice of the design matrix distribution has gained traction in the context of high-dimensional linear models.
In the context of the Lasso, \cite{korada_mont} established universality of the cost optimum, while for the elastic net, \cite{univ_elastic} established universality of the estimation risk. Similar results exist for random linear inverse problems (\cite{ASH19,OT18}). Recently, \cite{han2023universality} studied a broad spectrum of penalized regression problems and established universality of the optimum value of the empirical risk and the risk of regularized estimators under i.i.d.~designs. \cite{DLS2} established universality of the optimizer as well as their risk under a broader class of structured and/or semi-random covariate distributions. 

Parallel universality results exist for approximate message passing algorithms (cf.~\cite{BLM15,chen_lam,DLS1,wang2022universality}). Further, \cite{hu2022universality} 
demonstrated equivalence between regression problems where the rows of the design matrix arise from nonlinear transforms of Gaussian features and design matrices with linearized Gaussian features. Subsequently, \cite{saeed_mont} developed an extension of this result to the neural tangent model and a linear transformation of the i.i.d.~model.

We note that in terms of dependent covariates under the linear model framework, \cite{saeed_mont} is the most relevant to our setting. Although the two-layer neural tangent model and the random feature model presented there are intriguing examples of dependence from a machine learning standpoint, for this discussion, we  focus on the setup outlined in \cite[Section 3.3]{saeed_mont}, i.e., linear 
 functions of independent entries. Notably, this example involves rows of the form $\Sigma^{1/2} Z_i$, where the entries of $Z_i$ are independent sub-Gaussian. In this article, we broaden this understanding to encompass a different form of dependence with block structure (see Section \ref{sec_examples}) that prior work fails to capture. As noted earlier, this block dependence arises in various widely used  statistical models, as we demonstrate in Section \ref{sec_examples}. We also note another limitation of \cite{saeed_mont} in the context of estimation risk. The paper proves the universality of the training error, which is the same as our optimum but instead of estimation risk, discusses the universality of test error. Although estimation risk can be obtained in their setting by a similar argument, it requires some strong convexity conditions for the penalty function. In our work, we overcome this barrier. Leveraging methods from \cite{han2023universality} in conjunction with new proof ideas, we are able to  address the case of the Lasso where strong convexity does not hold. 

Our paper also builds a bridge between the proportional asymptotics literature and the nonparametric regression literature, which has a long history. In particular, we study additive models (see Section \ref{sec_examples}) that were previously analyzed in the high-dimensional setting with $\log(p)=o(n)$ (where $p$ quantifies the the number of component functions in the additive model) along with sparsity assumptions (see \cite{Tan_Zhang}) on the number of active component functions. \cite{Tan_Zhang} (see \cite{meier2009high,kol_yuan,yuan2016minimax,raskutti2012minimax} for similar models) considers both smoothness and sparsity based penalties and also allows for weak sparsity. We explore the additive models but under a high-dimensional setting where  $p/n\rightarrow \kappa$ ($\kappa$ can be greater than $1$), without imposing sparsity assumptions. Furthermore, while the prior literature was concerned with the rates of the estimation risk, we characterize the  exact behavior of the risk or rather the exact constant that it converges to in our setting. A similar setup arises while dealing with the functional regression model, where we obtain results without any sparsity assumption whatsoever on the underlying function class.
Finally, our paper considers a model from genomics where the covariates may be dependent but uncorrelated (see \cite{LD_gen}). We have described a toy model (see \cite{gene_hunt,knkf_1} for other models) in Section 
 \ref{GWAS_ex}, which showcases how such structures might arise and then showed that the model is amenable to our universality analysis.

\section{Proofs\label{sec_prf_otln}}

Our universality proof relies on the following key observation: for showing (i) universality of the optimum value of the empirical risk (Theorem \ref{thm_optimum}) as well as ii) universality of the estimation risk of regularized estimators (Corollary  \ref{est_err}), one requires to first establish that the estimators lie in some "nice" set with high probability. After this is established, points (i) and (ii) can be proved using separate techniques. 
For the subsequent discussion, we will use the term  universality of the optimizer to mean \eqref{eq:univoptim} in the statement of Theorem \ref{thm_optimizer}. Before we delve into our proof, we provide a brief description of our proof path below. 

As a first step for points (i) and (ii), we require to prove that the optimizer lies in a compact set whose $\ell_2$ norm is bounded by a constant and $\ell_{\infty}$ norm approaches $0$ in the large $n$ limit. We prove these key results in Lemmas \ref{ell2_bound_1} and \ref{ell_infty_bound1} respectively. Subsequently for part (i), we establish a version of normal approximation for linear combinations of dependent random variables. 
For part (ii), we draw inspiration from the techniques in  \cite{han2023universality}. That said, we emphasize that the techniques used in \cite{han2023universality} are heavily dependent on independence of the entries of the design, a property that does not hold in our setting. This raises additional technical challenges in our setting that we handle as we proceed. For part (ii), it suffices to prove Theorem \ref{thm_optimizer}, once we have established the aforementioned compactness property. For Theorem \ref{thm_optimizer} in turn, we require the compactness property, part (i), and a notion of separation  in the Gaussian model, which we name the Gordon gap.

As will be explained in the next section, the compactness property is non-trivial to establish in the dependent subGaussian setup, and we use a novel leave-d-out argument to solve the problem. In addition, while Stein's method has been used to show universality in several contexts, to the best of our knowledge, the dependency neighbourhood technique that we employ in this paper to show our  normal approximation as part of point (i) has never seen prior usage in the context of regression models with dependent designs. Finally, the Gordon gap is a technical tool that is only relevant to the independent Gaussian design problem (see Section \ref{sec_optimizer}) and has been used before in \cite{han2023universality,miolane_mont}. Nonetheless,  the scaling factor $\Lambda$, which is critical for covering our example in Section \ref{fr_ex}, poses significant additional challenges involving convergence of fixed points of certain equations that we overcome via an equicontinuity argument.

\subsection{Compactness property\label{opt_bound}}

As a first step, we prove that the optimizers $\hat{\varphi}^K_{\mathcal{X}}:=\hat{\beta}^K_{\mathcal{X}}-\beta_0$ belong to the following set with probability at least $1-n^{-c}$
\begin{equation}\mathcal{T}_n:=\{\varphi \in \mathbb{R}^p:\|\varphi\|_{\infty}\leq C\sqrt{\frac{d\log n}{n}},\|\varphi\|\leq C\}.\label{phi_Tn}\end{equation}
Since the ridge regression objective is strongly convex, the $\ell_2$ bound $\|\hat{\varphi}_{\mathcal{X}}^R\|\leq C$ follows trivially.  For the Lasso the analogous condition follows due to a restricted strong convexity type property that we discuss next. In the traditional theory for Lasso, in the setting of ultra-high dimensions ($p > > n$), it is customary to assume sparsity on the underlying signal. This allows one to define the familiar and famous restricted strong convexity condition (see \cite{negahban2012unified} and references cited therein). Using this, one can subsequently control the $\ell_2$ norm of the Lasso solution. Since we refrain from explicitly assuming sparsity assumptions on the signal, this strategy does not directly translate to our setting. However, in our high-dimensional regime, even when $p > n$, when the tuning parameter is sufficiently large (Assumption \ref{ass_3}, see also \cite[Lemma 6.3]{han2023universality}), one can establish that the Lasso solution is sparse (with sparsity proportional to the sample size) with high probability. This implies that having control on the smallest singular value of $\mathcal{X}_S$ (matrix formed by restricting $\mathcal{X}$ to a sparse subset) is sufficient for ensuring a restricted strong convexity type property of the Lasso solution. This allows one to bound  the $\ell_2$ norm of the Lasso solution.   

Establishing the $\ell_{\infty}$ bound poses challenges since the leave-one-method as outlined in \cite{han2023universality} fails in the presence of dependence in the data. Instead we use a novel leave-d-out method which essentially considers the following optimization problem
$$\hat{\varphi}^{(S)}=\argmin_{\varphi\in \mathbb{R}^p,\varphi_S=0}\frac{1}{2n}\|X\varphi-\xi\|^2+f(\varphi),$$
where $S$ is a subset of $1,\ldots,p$ of cardinality $d$. A detailed analysis together with our block dependence assumption then yields a suitable upper bound on $\|\hat{\varphi}^K_{\mathcal{X},S}\|$, which in turn provides the desired upper bound on $\|\hat{\varphi}^K_{\mathcal{X}}\|_{\infty}$. Such leave-d-out arguments have been used in prior literature in \cite{elkaroui_ptrf,Sur_2,Sur_4} in other contexts, however, these works do not use this for universality arguments.
For convenience we will show a scaled version of $\hat{\varphi}^K_{\mathcal{X}}$ (which we call $\hat{\theta}^K_{\mathcal{X}}$ and is defined below) belongs to $\mathcal{T}_n$ (see \eqref{phi_Tn}), from which it easily follows that $\hat{\varphi}^K_{\mathcal{X}}\in \mathcal{T}_n$. First we introduce some notations.

Recall the empirical risk minimization problem 
\begin{equation*}\hat{R}_n(Y,\mathcal{X},\beta)=\frac{1}{2n}\sum_{i=1}^n\|Y_i-\mathcal{X}^T_i\beta\|^2+\lambda f(\beta).
\end{equation*}
Since $\mathcal{X}$ is of the form $X\Lambda^{1/2}$, we let $\vartheta=\Lambda^{1/2} \beta$. Then the risk minimization problem can be rewritten as
\begin{equation*}\hat{R}_n(Y,X,\vartheta)=\frac{1}{2n}\sum_{i=1}^n\|Y_i-X^T_i\vartheta\|^2+\lambda f(\Lambda^{-1/2}\vartheta).
\end{equation*}
Since $\varphi=\beta-\beta_0$, we define $\theta=\vartheta-\vartheta_0$ so that $\theta=\Lambda^{1/2}\varphi$.
The new parametrization is concisely presented in the following display
\begin{equation}
\vartheta=\Lambda^{1/2} \beta, \quad,  \vartheta_0=\Lambda^{1/2} \beta_0, \quad \theta=\vartheta-\vartheta_0, \quad \Lambda^{-1/2} =\tilde{\Omega}.  \label{new_param}
\end{equation}
For brevity, we will also denote the value of $\hat{R}_n(Y,X,\vartheta)$ on using the variable $\theta$ as $R^K(X,\theta)$ (with K=L or R depending on whether $\ell_1$ or $\ell_2$ penalty is used), where we have suppressed the dependence on $\vartheta_0$ (and hence $Y$) and $n$. Finally, we define the following optimizers
\begin{align*}
\hat{\theta}^R_X&=\argmin_{\theta\in \mathbb{R}^p}\frac{1}{2n}\|X\theta-\xi\|^2+\frac{\lambda}{2}\|\tilde{\Omega}(\theta+\vartheta_0)\|^2,\\
\hat{\theta}^L_X&=\argmin_{\theta\in \mathbb{R}^p}\frac{1}{2n}\|X\theta-\xi\|^2+\frac{\lambda}{\sqrt{n}}\|\tilde{\Omega}(\theta+\vartheta_0)\|_1.
\end{align*}
 The following lemmas proved in the appendix show that the optimizers lie in the set $\mathcal{T}_n$ with high probability.
\begin{lemma}{\label{ell2_bound_1}}
     Under Assumptions \ref{ass_1}-\ref{ass_3} we have
    $$\|\hat{\theta}^K_{X}\|_{2}\leq C$$
with probability at least $1-n^{-c}$ for positive constants $c$ and $C$.
\end{lemma}

\begin{lemma}{\label{ell_infty_bound1}}
    Under Assumptions \ref{ass_1}-\ref{ass_3} we have
    $$\|\hat{\theta}^K_{X}\|_{\infty}\leq C\sqrt{\frac{d\log n}{n}}$$
with probability at least $1-n^{-c}$ for positive constants $c$ and $C$.
\end{lemma}

\subsection{Universality of the optimum}
We next describe our proof for the universality of the optimum that corresponds to point (i) discussed in the beginning of Section \ref{sec_prf_otln}. 
We may divide this proof into three steps. The first step replaces the minimum over the compact set with a minimum over a discrete set, and shows that the approximation error goes to zero. The second step involves approximating the minimization over the discrete set by a smoothed minimum, and showing once again that the approximation error is negligible. The final step proves that this smoothed minimum is universal, that is, its value is asymptotically the same if the dependent design is replaced by an i.i.d.~Gaussian design with matching moments.

This proof strategy has been used in prior work (see \cite[Theorem 2.3]{han2023universality} or \cite[Theorem 1]{saeed_mont}). The universality of the smoothed minimum is the crucial step and depends heavily on the dependence/independence structure in the underlying design matrix. It involves showing for any given set $S_n$ we have
\begin{equation}\sup_{\theta\in \mathcal{T}_n\cap S_{n}}|\E\left[\phi(X_i^\top\theta)\right]-\E\left[\phi(G_i^\top\theta)\right]|\rightarrow 0\label{pt_norm}\end{equation}
for a suitable choice of test function $\phi$, where $X_i$ is a row vector from the block dependent design and $G_i$ is a row vector from the independent Gaussian design. The arguments showing this step differs in our case compared to prior work. To establish such a pointwise normality condition, we use a version of Stein's method that involves dependency graphs together with the fact that the optimizer lies in the compact set $\mathcal{T}_n$.

Heuristically the universality  result holds because $X^T_i\hat{\beta}^K$ is a weighted sum of dependent sub-Gaussian random variables that converges to a Gaussian random variable if the dependence is ``low" in a suitable sense. It can be easily observed that the weighted sum may fail to converge if the weights are chosen arbitrarily, for example  
 if $\hat{\beta}^K=(1,0,\ldots,0)$, convergence to a Gaussian cannot be expected. For this reason we require the coordinates of $\hat{\beta}^K$ to be small enough, which justifies the $\ell_{\infty}$ bound proved in Lemma \ref{ell_infty_bound1}.

 \begin{proof}[Proof of Theorem \ref{thm_optimum}]

For this proof, we will assume that $\mathcal{S}_n$ is a subset of $\mathcal{T}_n$, which we can assume without loss of generality.\\

\textbf{Step 1: Discretization}\\

Pick a $\delta$ net of the set $\mathcal{S}_n$ and call it $S_{n,\delta}$. Let $\hat{\theta}=\argmin_{\theta\in S_{n}}R(\theta,X)$ and $\tilde{\theta}=\argmin_{\theta\in S_{n,\delta}}R(\theta,X)$. For the purposes of this proof, we write $\hat{\theta}^K_X$ as $\hat{\theta}$ and $R^K(\theta,X)$ as $R(\theta,X)$ since the same proof works for both Lasso and ridge regression. 
    We observe the following chain of inequalities
    \begin{align*}
        |\min_{\theta\in S_{n,\delta}}R(\theta,X)-\min_{\theta\in S_{n}}R(\theta,X)|&\leq \min_{\theta\in S_{n,\delta}}R(\theta,X)-\min_{\theta\in S_{n}}R(\theta,X)\\
        & \leq R(\tilde{\theta},X)-R(\hat{\theta},X)\\
        & \leq |R(\tilde{\theta},X)-R(\hat{\theta},X)|\\
        & \leq |\frac{1}{2n}(\|X\tilde{\theta}-\xi\|^2-\|X\theta-\xi\|^2|+|f(\tilde{\theta})-f(\theta)|\\
        &\leq \frac{1}{2n}\|X(\bar{\theta}-\tilde{\theta})\|^2+\frac{1}{n}\left|(X\bar{\theta}-\xi)^TX(\tilde{\theta}-\hat{\theta})\right|+C_1\|\tilde{\Omega}(\tilde{\theta}-\hat{\theta})\|\\
        &\leq \frac{1}{2n}\|X(\bar{\theta}-\tilde{\theta})\|^2+\frac{1}{n}|(\|X\|\|\bar{\theta}\|+\|\xi\|)\|X\|\|\tilde{\theta}-\hat{\theta})\|+C\|\tilde{\theta}-\hat{\theta}\|.
    \end{align*}
    Noting that $\|X\|\leq C\sqrt{n}$ and $\|\xi\|\leq C\sqrt{n}$ with probability at least $1-n^{-c}$ and $\max\{\|\theta\|,\|\tilde{\theta}\|\}\leq C$ (since $\mathcal{S}_n$ and $S_{n,\delta}$ are subsets of $\mathcal{T}_n$) we conclude that
    $$|\min_{\theta\in S_{n,\delta}}R(\theta,X)-\min_{\theta\in S_{n}}R(\theta,X)|\leq C\delta$$\\
with probability at least $1-n^{-c}$. For a bounded differentiable function with bounded Lipschitz derivative we have
\begin{equation}
    |E[\psi(\min_{\theta\in S_{n,\delta}}R(\theta,X))]-E[\psi(\min_{\theta\in S_{n}}R(\theta,X))]|\leq C\|\psi'\|_{\infty}\delta. \label{disc_app}
\end{equation}

    \textbf{Step 2: Approximation by smoothed minimum}\\
    
The smoothed minimum is defined as follows
    $$f_{\delta}(\alpha,X)=-\frac{1}{n\alpha}\log\sum_{\theta\in \mathcal{S}_{n,\delta}}\exp{(-n\alpha R(\theta,X))}$$
with the following approximation error (cf. \cite{saeed_mont})
    $$|f_{\delta}(\alpha,X)-\min_{\theta\in S_{n,\delta}}R(\theta,X)|\leq C(\delta)\alpha^{-1}.$$
 Again using a Lipschitz function as in \eqref{disc_app}, we obtain
    \begin{equation}
    |E[\psi(f_{\delta}(\alpha,X))]-E[\psi(\min_{\theta\in S_{n,\delta}}R(\theta,X)))]|\leq C\|\psi'\|_{\infty}\alpha^{-1}. \label{sm_app}
\end{equation}
    \textbf{Step 3: Universality of the smoothed minimum}\\
    
    The universality of the smoothed minimum follows via the interpolation path as shown in \cite{saeed_mont}. In particular one defines matrices $U_t$ where the rows are given by
$$U_{t,i}=sin(t)X_i+cos(t)G_i$$
and uses the relation
\begin{equation}|E(\psi(f_{\delta}(\alpha,X))-\psi(f_{\delta}(\alpha,G)))|=\left|E\int_{0}^{\pi/2}\frac{\partial}{\partial t}(\psi(f_{\delta}(\alpha,U_t))\right|.\label{sm_int_form}\end{equation}
It can be shown that $|E\frac{\partial}{\partial t}(\psi(f_{\delta}(\alpha,U_t))|\rightarrow 0$ once the pointwise normality condition (see \eqref{pt_norm}) is satisfied
 (see \cite[Lemma 4]{saeed_mont}), which we prove via Stein's method. Thus from \eqref{disc_app}, \eqref{sm_app}, and \eqref{sm_int_form}, we conclude that 
 $$\lim_{n\rightarrow \infty}|E[\psi(\min_{\theta\in S_{n}}R(\theta,X))]-E[\psi(\min_{\theta\in S_{n}}R(\theta,G))]|\rightarrow 0$$
 by suitably choosing $\delta\rightarrow0$, $\alpha\rightarrow\infty$, and $n \rightarrow\infty$, once we verify \eqref{pt_norm}. To show the latter, recall that $S_n$ can be assumed to a subset of $\mathcal{T}_n$ without loss of generality and hence we aim to show
$$\sup_{\theta\in S_{n}}|\E\left[\phi(X_i^\top\theta)\right]-\E\left[\phi(G_i^\top\theta)\right]|\rightarrow 0.$$
We use the following relation $$|\E\left[\phi(X_i^\top\theta)\right]-\E\left[\phi(G_i^\top\theta)\right]|=|\E\left[\phi(\nu X_i^\top\theta/\nu)\right]-\E\left[\phi(\nu Z)\right]|,$$
where $\nu^2=\|\theta\|^2$ and $Z\sim N(0,1)$.
	Define $$\psi(x)=e^{x^2/2}\int_{\infty}^{x}e^{-y^2/2}(\phi(\nu y)-\E((\phi(\nu Z)))dy,$$
which satisfies the Stein's identity
$$\psi'(x)-x\psi(x)=\phi(\nu x)-\E(\nu Z),$$
and if $\|\phi'(x)\|_{\infty}\leq B,$ then
\begin{equation}\max\{\psi(x),\psi'(x),\psi''(x)\}\leq 2\nu B\leq C,\label{deriv_bounds}\end{equation}
since $||\theta||\leq C_1$ (for some constant $C_1$) in our setting.
Thus it suffices to upper bound
$$\E\left[\psi'\left( \frac{X_i^\top\theta}{\nu}\right)- \frac{X_i^\top\theta}{\nu}\psi\left( \frac{X_i^\top\theta}{\nu}\right)\right].$$ 
By using \eqref{deriv_bounds}, the block dependent structure (Definition \ref{def_loc_dep}), and \cite[Theorem 3.6]{stein_survey} we have
$$\E\left[\psi'\left( \frac{X_i^\top\theta}{\nu}\right)- \frac{X_i^\top\theta}{\nu}\psi\left( \frac{X_i^\top\theta}{\nu}\right)\right]\leq \frac{Cd^2 \varsigma_3}{\sigma^{3}}+\frac{Cd^{3/2}\sqrt{\varsigma_4}}{\sigma^2},$$
with $$\varsigma_3=\frac{\max_jE|X_{ij}|^3\sum_{j=1}^p|\theta|_j^3}{\|\theta\|^{3}},\quad \text{ and }\varsigma_4=\frac{\max_jE|X_{ij}|^4\sum_{j=1}^p|\theta|_j^4}{\|\theta\|^{4}}.$$ Now suppose $\max_j|\theta_j|\leq \alpha_n$ and $\|\theta\|>\epsilon$, then with
$\varsigma=\max\{\max_jE|x_{ij}|^3,\max_jE|x_{ij}|^4\}$
we have
\begin{equation}\E\left[\psi'\left( \frac{X_i^\top\theta}{\nu}\right)- \frac{X_i^\top\theta}{\nu}\psi\left( \frac{X_i^\top\theta}{\nu}\right)\right]\leq \frac{d^2 \varsigma\alpha_n}{\epsilon}+\frac{d^{3/2}\sqrt{\varsigma}\alpha_n}{\epsilon}\label{stein_approx}.\end{equation}
Since $\theta\in S_n\subset \mathcal{T}_n$, $\alpha_n\rightarrow 0$ and hence the RHS of \eqref{stein_approx} goes to $0$ as $n\rightarrow \infty$. On the other hand, when $\|\theta\|<\epsilon$, we have
$$|\E\left[\phi(X_i^\top\theta)\right]-\E\left[\phi(G_i^\top\theta)\right]|\leq \|\phi^{'}\|_{\infty}(E(X_i^\top\theta)^2)^{1/2}+(E(G_i^\top\theta)^2)^{1/2}\leq C_1\epsilon,$$
for some $C_1$ since $\|\theta\|\leq C$ for $\theta \in \mathcal{T}_n$. Letting $\epsilon$ go to $0$, we have
\begin{align*}&\sup_{\theta\in S_{n}}|\E\left[\phi(X_i^\top\theta)\right]-\E\left[\phi(G_i^\top\theta)\right]|\\
\leq &\sup_{\theta\in S_{n}\cap\{\|\theta\|\geq \epsilon}\}|\E\left[\phi(X_i^\top\theta)\right]-\E\left[\phi(G_i^\top\theta)\right]|+\sup_{\theta\in S_{n}\cap\{\|\theta\|\leq \epsilon}\}|\E\left[\phi(X_i^\top\theta)\right]-\E\left[\phi(G_i^\top\theta)\right]| \rightarrow 0
\end{align*}
as desired.
\end{proof}

\textbf{Remark}: Note that in the proof we only require $d^{2}\alpha_n$ converges to $0$. So the result holds for growing $d$. In particular, $\alpha_n$ can be chosen as $(d\log n/n)^{1/2}$ and the proof works as long as $d^5\log n/n \rightarrow 0$.

\subsection{Universality of the optimizer \label{sec_optimizer}}
We use the terms universality of the optimizer and universality of the estimation risk interchangeably. This is because the proof of Theorem \ref{thm_optimizer} shows the universality of the optimizer and an application of Theorem \ref{thm_optimizer} gives Corollary \ref{est_err}, which deals with the universality of the estimation risk.

Our proof of this step builds upon \cite{han2023universality}.
The universality of the optimizer is established in \cite[Theorem 2.4]{han2023universality} under two conditions (called (O1) and (O2) in \cite{han2023universality}) with the regression setup considered in \cite[Theorem 3.1]{han2023universality}. We note that condition (O1) mentioned in there is the same as the bound on the $\ell_{\infty}$ norm of the optimizer, which we prove in Lemma \ref{ell_infty_bound1} by our leave-$d$-out method. Condition (O2) is the Gordon gap that states the following: 
if there is a difference in the unrestricted optimum and the optimum restricted to a particular set with high probability, then the optimizer does not belong to that set with high probability. In other words,  for any pair of positive constants $\psi^K$ and $z$, the following holds
\begin{equation}\{\inf_{\theta\in R^p}R^K(\theta,G)\leq \psi^K+z\}\cap\{\inf_{\theta\in S_n}R^K(\theta,G)\geq\psi^K+2z)\}\subset\{\hat{\theta}^K_G\in S^c_n\}.\label{inform_gord_gap}\end{equation}
We observe that the statement above concerns independent Gaussian design only i.e. we do not need to establish such a condition for the dependent sub-Gaussian design matrix. Although it may seem nontrivial that the universality of the optimizer follows from checking this condition for the Gaussian case, we show that (see proof of Theorem \ref{thm_optimizer}) the Gordon gap coupled with Theorem \ref{thm_optimum} (which shows universality of the optimum value of the empirical risk for all subsets of the nice compact set) readily yields the desired universality of the optimizer. This proof strategy was also utilized in \cite{han2023universality,saeed_mont}. Although our proof for the Gordon gap builds upon \cite[Theorems 3.4, 3.8]{han2023universality}, the scaling matrix $\Lambda$ calls for a far more nuanced analysis. In particular, Lemma \ref{Gordon_gap} below plays a critical role in the proof of Theorem \ref{thm_optimizer}. The scaling factor introduces considerable challenges in establishing this lemma. We  overcome this by employing a nontrivial equicontinuity argument (see Appendix \ref{univ_opt_appen}),
a step not necessitated in the case of \cite{han2023universality}.

Let $\omega_i=\tilde{\Omega}_{ii}$ and $\mu_0=\sqrt{n}\vartheta_0$ where $\tilde{\Omega}$ and $\vartheta$ are defined in \eqref{new_param}. Recall the fixed point solution $\beta^K_*,\gamma^K_*$ of the equation \eqref{eqn_system} and define $\theta_{*}^K=w_{*}^K/\sqrt{n}$ where
$$w_{*,i}^K=\eta_k\left(\mu_{0,i}+\gamma^K_*g_i,\frac{\gamma^K_*\lambda\omega_i^k}{\beta^K_*}\right)-\mu_{0,i}$$ with $k=1$ for $K=R$ and $k=2$ for $K=L$. Let us define the set $S_{n,\epsilon}$ as follows
$$\mathcal{S}_{n,\epsilon}:=\{\theta:|\varrho(\theta)-E(\varrho(\theta^K_{*}))|\geq \epsilon\}.$$ Henceforth we will omit the $\epsilon$ and call it $S_n$ since Theorem \ref{thm_optimizer} is proved for a fixed $\epsilon$). We have the following lemma.

\begin{lemma}{\label{Gordon_gap}}
    There exists positive constants $\psi^R$ and $\psi^L$ (denoted by $\psi^K$ with $K=R$ for ridge and $K=L$ for Lasso) and $z$ such that the following holds
    $$P\left(\inf_{\theta\in R^p}R^K(\hat{\theta},G)\geq \psi^K+z\right)\leq \delta_n \quad P\left(\inf_{\theta\in S_n}R^K(\hat{\theta},G)\leq \psi^K+2z)\right)\leq \delta_n$$
for some $\delta_n\rightarrow 0$.
\end{lemma}
Now we are ready to prove 
Theorem \ref{thm_optimizer}.
\begin{proof}[Proof of Theorem \ref{thm_optimizer}]

We start with formalizing the statement in \eqref{inform_gord_gap}. We have the following set of inequalities for suitable choices of $\gamma_n$, $\delta_n$, $\epsilon_n \rightarrow 0$.

\begin{align*}
    P(\hat{\theta}^K_X\in \mathcal{S}_n)&\leq P(\hat{\theta}^K_X\in \mathcal{S}_n\cap \mathcal{T}_n)+\epsilon_n\\
    & \leq P\left(\inf_{\theta\in \mathcal{T}_n}R^K(\theta,X)\geq \psi^K+3z\right)+P\left(\inf_{\theta\in \mathcal{S}_n\cap \mathcal{T}_n}R^K(\theta,X)\leq \psi^K+6z\right)+\epsilon_n\\
    &\leq P\left(\inf_{\theta\in \mathcal{T}_n}R^K(\theta,G)\geq \psi^K+z\right)+P\left(\inf_{\theta\in \mathcal{S}_n\cap \mathcal{T}_n}R^K(\theta,G)\leq \psi^K+2z\right)+\epsilon_n+2\gamma_{n}\\
    &\leq P\left(\inf_{\theta}R^K(\theta,G)\geq \psi^K+z\right)+P\left(\inf_{\theta\in \mathcal{S}_n }R^K(\theta,G)\leq \psi^K+2z\right)+3\epsilon_n+2\gamma_{n}\\
    &\leq 2\delta_n+3\epsilon_n+2\gamma_{n}.\\
\end{align*}
We note that the first and fourth inequality follows from Lemmas \ref{ell2_bound_1} and \ref{ell_infty_bound1}. The second inequality follows from the following inclusion
$$\{\inf_{\theta\in \mathcal{T}_n}R^K(\theta,X)\leq \psi^K+3z\}\cap \{\inf_{\theta\in \mathcal{S}_n\cap \mathcal{T}_n}R^K(\theta,X)\geq \psi^K+6z\}\subset \{\hat{\theta}^K_X\in (\mathcal{S}_n\cap \mathcal{T}_n)^c\}.$$
The third inequality follows from Theorem \ref{thm_optimum} with a suitable choice of test function, while the last inequality follows from Lemma \ref{Gordon_gap}.
Next we use \eqref{new_param} to obtain our statements with the old parameters. For $\varphi$, we use the fact that $\varphi=\tilde{\Omega}\theta$ and $\theta=w/\sqrt{n}$ to obtain
$$\varphi^K_{*}=\tilde{\Omega} \frac{w_{*}^K}{\sqrt{n}}$$ so that
\begin{align}\varphi^R_{*,i}&=\eta_2\left(\beta_{0,i}+\frac{\gamma^R_*\omega_ig_i}{\sqrt{n}},\frac{\gamma^R_*\lambda\omega_i^2}{\beta^R_*}\right)-\beta_{0,i}\\
\varphi^L_{*,i}&=\eta_1\left(\beta_{0,i}+\frac{\gamma^L_*\omega_ig_i}{\sqrt{n}},\frac{\gamma^L_*\lambda\omega_i^2}{\sqrt{n}\beta^L_*}\right)-\beta_{0,i},
\end{align}
where we have used the fact that
$$\eta_1(\alpha x,\alpha\lambda)=\alpha\eta_1( x,\lambda),\quad \eta_2(\alpha x,\lambda)=\alpha\eta_2( x,\lambda).$$
Since $\hat{\varphi}^K_{X}=\tilde{\Omega}\hat{\theta}^K_X$, where the entries of $\tilde{\Omega}$ are bounded, we can easily redefine $\varrho$ so that
\begin{equation}P(|\varrho(\hat{\varphi}^K_X)-E(\varrho(\varphi^K_{*}))|\geq \epsilon)\leq 2\delta_n+2\epsilon_n+3\gamma_{n}.\label{subg-bound}\end{equation}

Now we observe that although the above result is proved for dependent subGaussian designs, the latter subsumes the case of independent Gaussian design and hence it trivially follows that
\begin{equation}P(|\varrho(\hat{\varphi}^K_G)-E(\varrho(\varphi^K_{*}))|\geq \epsilon)\leq 2\delta_n+3\epsilon_n+2\gamma_{n}.\label{Gauss-bound}\end{equation}
Combining \eqref{subg-bound} and \eqref{Gauss-bound}, we obtain that 
$$P(|\varrho(\hat{\varphi}^K_X)-\varrho(\hat{\varphi}^K_G)|\geq 2\epsilon)\leq 4\delta_n+6\epsilon_n+4\gamma_{n}.$$
Defining $4\delta_n+6\epsilon_n+4\gamma_{n}$ as $\varepsilon_n$  and replacing $2\epsilon$ by $\epsilon$ we observe that the desired claim is proved.
\end{proof}
We now have all ingredients in place for the proof of Corollary \ref{est_err}.
\begin{proof}[Proof of Corollary \ref{est_err}]
   Choose $\varrho$ as the function $\|.\|^2$. Although this function is not Lipshitz in the entire domain $\mathbb{R}^p$ it is Lipshitz in the bounded domain $\{\varphi:\|\varphi\|\leq C\}$. Since our optimizers have bounded $\ell_2$ norm with probability at least $1-n^{-c}$, we work in the domain $\{\varphi:\|\varphi\|\leq C\}$ for our purposes and use Theorem \ref{thm_optimizer} to conclude
   \begin{align*}\|\hat{\varphi}^R_{X}\|^2-E_{\bar{Q}_n}\left[\eta_2\left(M+\frac{\gamma^R_*\Omega Z}{\sqrt{n}},\frac{\gamma^R_*\lambda\Omega^2}{\beta^R_*}\right)-M\right]^2&\xrightarrow{P}0,\\
\|\hat{\varphi}^L_{X}\|^2-E_{\bar{Q}_n}\left[\eta_1\left(M+\frac{\gamma^L_*\Omega Z}{\sqrt{n}},\frac{\gamma^L_*\lambda\Omega^2}{\sqrt{n}\beta^L_*}\right)-M\right]^2&\xrightarrow{P}0.
\end{align*}
Substituting $\hat{\varphi}^K_{X}=\hat{\beta}^K_{X}-\beta_0$, we conclude our proof.
\end{proof}

\section{Discussion \label{sec_disc}}
In this paper, we established universality results within the framework of empirical risk minimization, demonstrating that the optimal value of the regularized empirical risk and estimation risk of regularized estimators converge to the same value for block dependent sub-Gaussian designs, as in the case of i.i.d.~Gaussian designs. In sum, to the best of the authors' knowledge, our paper is the first in the literature that handles dependence structures beyond correlation-based dependence or those allowed by right rotationally invariant designs. Our examples from Section \ref{sec_examples} demonstrate the importance of studying such general dependencies, providing the first segue into nonparametric problems. Notably, this allows us to characterize the (asymptotically) exact risk of estimators for popular non-parametric regression models under high-dimensional covariates. We conclude our manuscript with a couple of directions for future work.

First, our results do not account for dependence captured by an explicit covariance matrix. A logical extension of our work would involve proving dependence when feature vectors take the form $\Sigma^{1/2}Z_i$. Here, instead of $Z_i$ being a random vector with independent sub-Gaussian entries, as considered in \cite{saeed_mont}, $Z_i$ is isotropic, and block dependent sub-Gaussian vectors.
We conjecture that under some regularity assumptions on the covariance matrix $\Sigma$, the optimal value of the risk and the risk of regularized estimators should be asymptotically the same as that of the design with rows of the form $\Sigma^{1/2}G_i$ where the entries of $G_i$ are i.i.d.~Gaussian. 
While proving universality of the optimum in this scenario aligns with the approach in \cite{saeed_mont}, generalizing our leave-d-out method in presence of a general $\Sigma$ and subsequently proving universality of the optimizer presents nontrivial challenges, which we defer to future work.

Second, an intriguing avenue involves considering a more general dependence pattern, where each entry of the feature vector depends on only a small number of other entries, without a block structure. Currently, the block structure is pivotal in our proofs of $\ell_{\infty}$ bounds via our leave-d-out method. We view our current contribution as a stepping stone towards investigating these other notions of dependence. The absence of the block structure implies a stronger form of dependence, necessitating new tools. We defer this to future investigation.  

\section{Acknowledgements}
The authors would like to thank Nabarun Deb, Rounak Dey, and Subhabrata Sen for helpful discussions at various stages of the manuscript. P.S.~would like to acknowledge partial funding from NSF DMS-2113426.

\bibliography{univ_subG}
\bibliographystyle{alpha}.
\newpage
\appendix
\section{Proof of the Lemmas}
First we state a concentration result which will be used in proving the theorems.
\subsection{Concentration result}

\begin{lemma}{\label{ver_lemma}}
   Let $X$ be an $n\times p$ matrix whose rows $X_i$ are independent sub-Gaussian isotropic random vectors in $\mathbb{R}^n$. Then for every $t\geq 0$, with probability at least $1-2\exp(-c_Kt^2)$ one has
   \begin{equation}
       \sqrt{n}-C_K\sqrt{p}-t\leq s_{\min}(X)\leq s_{\max}(X)\leq \sqrt{n}+C_K\sqrt{p}+t,
   \end{equation}\label{row_indep}
where $s_{\min}(X)$ and $s_{\max}(X)$ are the smallest and largest singular values of the matrix $X$. The constant s $C_K,c_K$ depend only on the sub-Gaussian norm $K=\max_i\|X_i\|_{\psi_2}$ of the rows.

\end{lemma}
\begin{proof}
  See \cite{vershynin_2012}.  
\end{proof}
\subsection{Bounds on the optimizers}

Throughout this section we will work with Assumptions \ref{ass_1} and \ref{ass_3}. Further we use the following modified assumption instead of \ref{ass_2} to account for change in scale (see \eqref{new_param})
\begin{assumption}{\label{ass_4}}The signal satisfies the bound $\|\vartheta_0\|\leq C$. Let $v_i=\sqrt{n}\vartheta_{0,i}$ for $1\leq i\leq p$. Then we assume that
$$\frac{1}{p}\sum_{i=1}^p\delta_{\lambda_i,v_i}\xRightarrow{W_2}\mu.$$
Further we assume that $\|v\|_{\infty}\leq C$.\end{assumption}

\begin{proof}[Proof of lemma \ref{ell2_bound_1}]
Recall that we work under the proportional asymptotic setup i.e. $\lim_{n\rightarrow \infty}p/n=\kappa$. For the sake of clarity we will first deal with the setup with  $\kappa^{-1}>C_0> C_K$ for some $C_0$ where $C_K$ is the constant appearing in Lemma \ref{ver_lemma}. We will later show that in the general setup similar bounds can be achieved by suitably modifying the tuning parameter $\lambda$.\\

\textbf{Case 1: $n/p\rightarrow \kappa^{-1}>C_0>C_K^2$}\\

\textbf{Ridge regression}\\

Recall that $\tilde{\Omega}=\Lambda^{-1/2}$. By the optimality of $\hat{\theta}_X^R$, it easily follows that
$$\frac{1}{2}\|X\hat{\theta}_X^R-\xi\|^2+\frac{n\lambda}{2}\|\tilde{\Omega}(\hat{\theta}_X^R+\vartheta_0)\|^2\leq \frac{1}{2}\|\xi\|^2+\frac{n\lambda}{2}\|\tilde{\Omega}\vartheta_0\|^2,$$
whence $\|X\hat{\theta}^R_X\|^2\leq C(\frac{1}{2}\|\xi\|^2+\frac{n\lambda}{2}\|\tilde{\Omega}\vartheta_0\|^2)$.

Using a lower bound on the singular value of $X$ from Lemma \ref{ver_lemma} we observe that the left-hand side of the last inequality is lower bounded by $cn\|\hat{\theta}^R_X\|^2$ with probability at least $1-n^{-c}$ by choosing $t=c_1\sqrt{\log n}$ (for suitable choice of $c_1$). Similarly using $E[\|\xi\|^2]\leq Cn$ and sub-Gaussian concentration inequality for the random vector $\xi$ we easily conclude that $\|\hat{\theta}^R_X\|\leq C$ with probability at least $1-n^{-c}$. Note that we have used $\|\vartheta_0\|^2\leq C$ and that $\tilde{\Omega}$ is a diagonal matrix with bounded entries.\\

\textbf{Lasso}\\

We have an analogous bound for Lasso penalty. Indeed, similar to the above argument we obtain

$$\|X\hat{\theta}^L_X\|^2\leq C\left(\frac{1}{2}\|\xi\|^2+\frac{n\lambda}{\sqrt{n}}\|\tilde{\Omega}\vartheta_0\|_1\right).$$

Since entries of $\tilde{\Omega}$ are bounded by some constant $C$ and $\|\vartheta_0\|\leq C$, we get $\|\tilde{\Omega}\vartheta\|_1\leq C \sqrt{p}\|\vartheta_0\|\leq  C \sqrt{p}$. Using the fact that $\|\xi\|\leq C\sqrt{n}$ with probability at least $1-n^{-c}$ we obtain
$\|X\hat{\theta}^L_X\|^2\leq Cn$ (we have used that $(p/n)^{1/2}\leq C$) with probability at least $1-n^{-c}$. A lower bound on the singular value of $X$ from Lemma \ref{ver_lemma} then yields $\|\hat{\theta}^L_X\|\leq C.$\\

\textbf{Case 2:} $n/p \rightarrow \kappa^{-1} \leq C^2_K$\\

\textbf{Ridge regression}\\

This is a relatively simpler case since we can directly analyze the closed form expression of the ridge estimator and use strong convexity. Indeed for the ridge estimator $\hat{\vartheta}_X^R$, we have

$$\|\hat{\vartheta}_X^R\|= \|(X^TX+n\lambda\tilde{\Omega}^2)^{-1}(X^TX\vartheta_0+X^T\xi)\|\leq (n\lambda c)^{-1}(\|X^TX\|\|\vartheta_0\|+\|X^T\|\|\xi\|).$$

By lemma \ref{row_indep} we can upper bound $\|X^TX\|$ and $\|X\|$ by $n$ and $\sqrt{n}$ respectively. Since $\|\vartheta_0\|$ is bounded by a constant and $\|\xi\|\leq C\sqrt{n}$ with probability at least $1-n^{-c}$ we obtain $\|\hat{\vartheta}_X^R\|\leq C$ with probability at least $1-n^{-c}$ which in turn implies $\|\hat{\theta}_X^R\|\leq C$ with probability at least $1-n^{-c}$.\\

\textbf{Lasso}\\

 This case is more difficult because when the condition $n/p \rightarrow \kappa^{-1} > C^2_K$ is not satisfied, the bound on the smallest singular value (see Lemma 
 \ref{ver_lemma}) does not work. In fact, for $p>n$ the smallest singular value is zero. However, we can follow the idea outlined in \cite{han2023universality} -  ensuring that $\lambda$ is larger than a certain value one can ensure that the optimizer is sparse with the number of nonzero components a small fraction of $n$. Restricting ourselves to such a set $S$, it is then enough to give a lower bound for the smallest singular value of $X_S$.

The following lemma shows that for large enough $\lambda$ the optimizer is sparse. 

\begin{lemma}{\label{Lasso_sparse}}
    For any $c_0<1$ There exist a $\lambda_0$ which depends on $c_0$, such that for $\lambda\geq \lambda_0$,the solution $\hat{\theta}^L_{X}$ is $(c_0n)$ sparse with probability at least $1-n^{c(\lambda,c_0)}$\label{sparse_singular_lemma}
    where $c(\lambda,c_0)$ is a constant depending on $c_0$ and the tuning parameter $\lambda$.

\end{lemma}
We defer the proof of this lemma for the moment. Now we can easily derive the $\ell_2$ bound for the Lasso case. By the optimality of $\hat{\theta}^L_X$, it easily follows that
$$\frac{1}{2}\|X\hat{\theta}_X^L-\xi\|^2+\frac{\sqrt{n}\lambda}{2}\|\tilde{\Omega}(\hat{\theta}_X^L+\vartheta_0)\|_1\leq \frac{1}{2}\|\xi\|^2+\frac{\sqrt{n}\lambda}{2}\|\tilde{\Omega}\vartheta_0\|_1.$$
We can rewrite the inequality using $\vartheta$ instead of $\theta$ as follows:
$$\frac{1}{2}\|X\hat{\vartheta}^L_X-Y\|^2+\frac{\sqrt{n}\lambda}{2}\|\tilde{\Omega}\hat{\vartheta}^L_X\|_1\leq \frac{1}{2}\|\xi\|^2+\frac{\sqrt{n}\lambda}{2}\|\tilde{\Omega}\vartheta_0\|_1.$$
Rearranging the terms it is not difficult to observe that 
$$\frac{1}{2}\|X\hat{\vartheta}^L_X\|^2\leq C\left(\frac{1}{2}\|\xi\|^2+\frac{\sqrt{n}\lambda}{2}\|\tilde{\Omega}\vartheta_0\|_1+\|Y\|^2\right).$$
The right hand side is less than equal to $n$ due to sub-Gaussianity assumptions. Indeed $\|\xi\|^2\leq n$ with high probability and the same holds for the term $\|X\vartheta_0\|^2$ using Lemma \ref{row_indep} and boundedness of $\|\vartheta_0\|$. These two bounds can be used to show that $\|Y\|^2\leq Cn$ with high probability. Also, using the boundedness of the entries of $\tilde{\Omega}$ and Cauchy-Schwarz inequality we have $\frac{\sqrt{n}\lambda}{2}\|\tilde{\Omega}\vartheta_0\|_1\leq C\sqrt{np}\|\vartheta_0\|\leq Cn$. From these bounds the claimed bound on the RHS follows. On the other hand the LHS is greater than $cn\|\hat{\vartheta}^L_X\|^2$ with high probability by Lemma \ref{sparse_singular_lemma}. Thus we obtain $\|\hat{\vartheta}^L_X\|\leq C$ for some constant $c$ with high probability. Since $\|\vartheta_0\|\leq C$, we obtain that $\|\hat{\theta}^L_X\|\leq C$ with high probability.

\end{proof}
\begin{proof}[Proof of Lemma \ref{ell_infty_bound1}]

\textbf{Case 1: $n/p\rightarrow \psi^{-1}>C_0>C_K^2$}\\

\textbf{Ridge Regression}\\

Let $A=X/\sqrt{n}$, $w=\sqrt{n}\theta$, $\hat{w}^R_{A}=\sqrt{n}\hat{\theta}^R_{X}$. We also denote $\mu_0=\sqrt{n}\vartheta_0$ and write $\hat{w}^R_{A}$ as $\hat{w}$ for the sake of brevity. Thus it is enough to show that $\|\hat{w}\|_{\infty}\leq \sqrt{d\log n}$ with high probability.
We consider the column-leave-$d$-out version of the problem. Let $S$ be a given subset of $\{1,\ldots,p\}$.
$$\hat{w}^{(S)}=\argmin_{w\in \mathbb{R}^p,w_S=0}\frac{1}{2}\|Aw-\xi\|^2+\frac{\lambda}{2}\|\tilde{\Omega}(\mu_0+w)\|^2.$$
Since $\hat{w}$ is the optimal solution we obtain the following set of inequalities
\begin{align*}
    0&\leq \frac{1}{2}\|A\hat{w}^{(S)}-\xi\|^2+\frac{\lambda}{2}\|\tilde{\Omega}(\mu_0+\hat{w}^{(S)})\|^2-\frac{1}{2}\|A\hat{w}-\xi\|^2-\frac{\lambda}{2}\|\tilde{\Omega}(\mu_0+\hat{w})\|^2\\
    & = -\frac{1}{2}\|A(\hat{w}^{(S)}-\hat{w})\|^2+\langle A(\hat{w}^{(S)}-\hat{w}),A\hat{w}^{(S)}-\xi)\rangle +\frac{\lambda}{2}\|\tilde{\Omega}(\mu_0+\hat{w}^{(S)})\|^2 - \frac{\lambda}{2}\|\tilde{\Omega}(\mu_0+\hat{w})\|^2.
\end{align*}
We have the following decomposition 
\begin{align*}\langle A(\hat{w}^{(S)}-\hat{w}),A\hat{w}^{(S)}-\xi)\rangle&=-\langle A_S \hat{w}_S,(A_{-S}\hat{w}^{(S)}_{-S}-\xi)\rangle +\langle A_{-S}(\hat{w}^{(S)}_{-S}-\hat{w}_{-S}),(A_{-S}\hat{w}^{(S)}_{-S}-\xi)\rangle\\
&=-\langle A_S \hat{w}_S,(A_{-S}\hat{w}^{(S)}_{-S}-\xi)\rangle -\lambda\langle\hat{w}^{(S)}_{-S}-\hat{w}_{-S},\tilde{\Omega}^2_{-S}(\mu_{0,-S}+\hat{w}^{(S)}_{-S})\rangle,
\end{align*}
where the last equality uses the KKT condition
$A_{-S}^T(A_{-S}\hat{w}^{(S)}-\xi)=-\lambda\tilde{\Omega}^2_{-S}(\mu_{0,-S}+\hat{w}^{(S)}_{-S})$.
Thus we have
\begin{align*}
    0\leq& -\frac{1}{2}\|A(\hat{w}^{(S)}-\hat{w})\|^2-\langle A_S \hat{w}_S,(A_{-S}\hat{w}^{(S)}_{-S}-\xi)\rangle\\
    &-\lambda\langle\hat{w}^{(S)}_{-S}-\hat{w}_{-S},\tilde{\Omega}^2_{-S}(\mu_{0,-S}+\hat{w}^{(S)}_{-S})\rangle+\frac{\lambda}{2}\|\tilde{\Omega}(\mu_0+\hat{w}^{(S)})\|^2 - \frac{\lambda}{2}\|\tilde{\Omega}(\mu_0+\hat{w})\|^2 \\
    & \leq -\frac{1}{2}\|A(\hat{w}^{(S)}-\hat{w})\|^2-\langle A_S \hat{w}_S,(A_{-S}\hat{w}^{(S)}_{-S}-\xi)\rangle+\frac{\lambda}{2}\|\tilde{\Omega}_S\mu_{0,S}\|^2 - \frac{\lambda}{2}\|\tilde{\Omega}_S(\mu_{0,S}+\hat{w}_{S})\|^2,
\end{align*}
where the last inequality follows from the convexity of the function $\|.\|^2$.
Thus we obtain
$$\frac{1}{2}\|A(\hat{w}^{(S)}-\hat{w})\|^2+\frac{\lambda}{2}\|\tilde{\Omega}_S\hat{w}_{S}\|^2\leq \|\hat{w}_{S}\|(\|A^T_{S}(A_{-S}\hat{w}^{(S)}_{-S}-\xi)\|+\lambda\|\tilde{\Omega}^2_S\mu_{0,S}\|).$$
From this relation, it easily follows that for some constant $c$
\begin{equation}c\|\hat{w}_S\|\leq 2\lambda^{-1}(\|A^T_{S}(A_{-S}\hat{w}^{(S)}_{-S}-\xi)\|+\lambda\|\tilde{\Omega}^2_S\mu_{0,S}\|)\leq 2\lambda^{-1}(\|A^T_{S}A_{-S}\hat{w}^{(S)}_{-S}\|+\|A^T_{S}\xi\|+\lambda\|\tilde{\Omega}^2_S\mu_{0,S}\|).\label{loo_ridge}\end{equation}
By the optimality of $\hat{w}^{(S)}$, we have
$$\frac{1}{2}\|A_{-S}\hat{w}_{-S}^{(S)}-\xi\|^2+\frac{\lambda}{2}\|\tilde{\Omega}(\mu_0+\hat{w}_{-S}^{(S)})\|^2\leq \frac{1}{2}\|\xi\|^2+\frac{\lambda}{2}\|\tilde{\Omega}\mu_0\|^2,$$
whence $\|A_{-S}\hat{w}_{-S}^{(S)}\|^2\leq C\left(\frac{1}{2}\|\xi\|^2+\frac{\lambda}{2}\|\tilde{\Omega} \mu_0\|^2\right)$.
Since $\|\xi\|^2\leq Cn$ with probability at least $1-n^{-c}$ and $\|\tilde{\Omega}\mu_0\|^2 \leq Cn$, we obtain $\|A_{-S}\hat{w}_{-S}^{(S)}\|^2\leq Cn$ with probability at least $1-n^{-c}$.

Now for a single coordinate $s$ there exist a set $S$ 
 (with $|S|\leq d$) such that $s\in S$ and the columns of $A_S$ and the columns of $A_{-S}$ are independent. Let $i\in S$. Using the facts that the elements of $A_i$ are independent subGaussian and that $A_{i}$ and $A_{-S}\hat{w}^{(S)}_{-S}$ are independent, we have 
\begin{equation}\|A^T_{S}A_{-S}\hat{w}^{(S)}_{-S}\|\leq \sqrt{d}\max_{i\in S}|A^T_{i}A_{-S}\hat{w}^{(S)}_{-S}|\leq \sqrt{d\log n}\label{cross-term-bound}\end{equation}
holds with probability at least $1-n^{-c}$ conditioned on the event $\|A_{-S}\hat{w}_{-S}^{(S)}\|^2\leq Cn$ which itself holds with probability $1-n^{-c}$.

 Since the elements of $A_i$ are i.i.d.~and $A_i$ and $\xi$ are independent,$|A^{T}_i\xi|\leq \sqrt{\log n}$ holds with probability at least $1-n^{-c}$. Indeed conditioning on the event $\|\xi\|\leq C\sqrt{n}$ the result follows from the sub-Gaussian concentration and the event itself holds with probability at least $1-n^{-c}$. Similar to the bounds obtained for $\|A^T_{S}A_{-S}\hat{w}^{(S)}_{-S}\|$ we obtain $\|A^{T}_{S}\xi\|\leq \sqrt{d\log n}$ with probability at least $1-n^{-c}$. Finally since the entries of the signal $\mu_0$ are bounded (recall that the entries of $\vartheta_0$ are of the order $O(p^{-1/2})$ by Assumption \ref{ass_4}), using \eqref{loo_ridge} we have the following inequality
$$|\hat{w}_s|\leq \|\hat{w}_S\|\leq \sqrt{d\log n},$$
with probability at least $1-n^{-c}$.
 
 Also this event holds uniformly in $s$ since we can choose $c>1$ and use the union bound.  Hence $$\|\hat{w}^R\|_{\infty}\leq \sqrt{d \log n},$$ with probability at least $1-n^{-c}$.\\

\textbf{Lasso}
\\

Once again we start with the leave-d-out estimator. Let $S$ be a given subset of $\{1,\ldots,p\}$ and define:
$$\hat{w}^{(S)}=\argmin_{w\in \mathbb{R}^p,w_S=0}\frac{1}{2}\|Aw-\xi\|^2+\lambda\|\tilde{\Omega}(\mu_0+w)\|_1.$$
The KKT condition is given by 
$$A_{-S}^T(A_{-S}\hat{w}^{(S)}_{-S}-\xi)+\lambda \tilde{\Omega}_{-S}v_{-S}=0,$$
where $v_{-S}$ is the subgradient of $\|.\|_1$ at $\tilde{\Omega}(\mu_0+\hat{w}^{(S)})$. We write the risk $R(Y,A,\vartheta)$ as a function of $w$ i.e. $R(w)$ suppressing the other variables. Denoting $\hat{w}$ as the Lasso optimizer we have
\begin{align*}
    R(0,\hat{w}_{-S})=&\frac{1}{2}\|A_{-S}\hat{w}_{-S}-\xi\|^2+\lambda\|\tilde{\Omega}(\hat{w}_{-S}+\mu_0)\|_1\\
    =&\frac{1}{2}\|A_{-S}(\hat{w}_{-S}-\hat{w}^{(S)}_{-S})\|^2+\frac{1}{2}\|A_{-S}\hat{w}^{(S)}_{-S}-\xi\|^2+\langle A_{-S}(\hat{w}_{-S}-\hat{w}^{(S)}_{-S}),A_{-S}\hat{w}^{(S)}_{-S}-\xi\rangle\\
    &+\lambda\|\tilde{\Omega}(\hat{w}^{(S)}_{-S}+\mu_0)\|_1+\lambda\|\tilde{\Omega}(\hat{w}_{-S}+\mu_0)\|_1-\lambda\|\tilde{\Omega}(\hat{w}^{(S)}_{-S}+\mu_0)\|_1\\
    = &R(0,\hat{w}^{(S)}_{-S})+\frac{1}{2}\|A_{-S}(\hat{w}_{-S}-\hat{w}^{(S)}_{-S})\|^2+\langle A_{-S}(\hat{w}_{-S}-\hat{w}^{(S)}_{-S}),A_{-S}\hat{w}^{(S)}_{-S}-\xi\rangle\\ 
    &+\lambda\|\tilde{\Omega}(\hat{w}_{-S}+\mu_0)\|_1-\lambda\|\tilde{\Omega}(\hat{w}^{(S)}_{-S}+\mu_0)\|_1\\
    = &R(0,\hat{w}^{(S)}_{-S})+\frac{1}{2}\|A_{-S}(\hat{w}_{-S}-\hat{w}^{(S)}_{-S})\|^2-\langle \hat{w}_{-S}-\hat{w}^{(S)}_{-S},\lambda \tilde{\Omega}_{-S}v_{-S}\rangle\\
    &+\lambda\|\tilde{\Omega}(\hat{w}_{-S}+\mu_0)\|_1-\lambda\|\tilde{\Omega}(\hat{w}^{(S)}_{-S}+\mu_0)\|_1\\
    \geq &R(0,\hat{w}^{(S)}_{-S})+\frac{1}{2}\|A_{-S}(\hat{w}_{-S}-\hat{w}^{(S)}_{-S})\|^2.
\end{align*}
The fourth equality follows from the KKT condition and the last inequality from the convexity of the function $\|.\|_1$.
Define $f_1(x):=R((x,\hat{w}_{-S}))$ and $f_2(x):=R((x,\hat{w}^{(S)}_{-S}))$. Then we have shown 
\begin{equation}\frac{1}{2}\|A_{-S}(\hat{w}_{-S}-\hat{w}^{(S)}_{-S})\|^2\leq f_1(0)-f_2(0)\leq (f_1(0)-\min_x f_1(x))-(f_2(0)-\min_x f_2(x))\label{f_bound}.\end{equation}
The last inequality follows from the fact $$\min_x f_1(x)= R(\hat{w})\leq 
 \min_x R((x,\hat{w}^{(S)}_{-S}))=\min_x f_2(x).$$
We expand $f_1(x)$ as follows:
\begin{align*}
    f_1(x)=&\frac{1}{2}\|A_Sx+A_{-S}\hat{w}_{-S}-\xi\|^2+\lambda\|\tilde{\Omega}_S(x+\mu_{0,S})\|_1+\lambda\|\tilde{\Omega}_{-S}(\hat{w}_{-S}+\mu_{0,-S})\|_1\\
    =& \frac{1}{2}\|A_{-S}\hat{w}_{-S}-\xi\|^2+ \frac{1}{2}\|A_Sx\|^2+\langle A_Sx, A_{-S}\hat{w}_{-S}-\xi\rangle\\
    & +\lambda\|\tilde{\Omega}_S(x+\mu_{0,S})\|_1+\lambda \|\tilde{\Omega}_{-S}(\hat{w}_{-S}+\mu_{0,-S})\|_1+\lambda\|\tilde{\Omega}_S\mu_{0,S}\|_1-\lambda \|\tilde{\Omega}_S\mu_{0,S}\|_1\\    
    = & f_1(0)+\frac{1}{2}\|A_Sx\|^2+\langle A_Sx, A_{-S}\hat{w}_{-S}-\xi\rangle +\lambda\|\tilde{\Omega}_S(x+\mu_{0,S})\|_1-\lambda \|\tilde{\Omega}_S\mu_{0,S}\|_1.
\end{align*}
Let $y_1=A_{-S}\hat{w}_{-S}-\xi$ and $y_2=A_{-S}\hat{w}^{(S)}_{-S}-\xi$. Also define
$$g(y)=\min_x \frac{1}{2}\|A_Sx\|^2+\langle A_Sx, y\rangle + +\lambda\|\tilde{\Omega}_S(x+\mu_{0,S})\|_1.$$
Then we observe that 
$$f_1(0)-f_2(0)\leq (f_1(0)-\min_x f_1(x))-(f_2(0)-\min_x f_2(x))=g(y_2)-g(y_1).$$
Further we define the following functions:
\begin{align*}
    h(x)&=\frac{1}{2}\|A_Sx\|^2+\langle A_Sx, y\rangle\\
    b(x)&=\lambda\|\tilde{\Omega}_S(x+\mu_{0,S})\|_1.
\end{align*}
Thus we have,
\begin{align*}
g(y)&=\min_x h(x)+ b(x)\\
&= \sup_x -h^*(x)-b^*(-x),
\end{align*}
where $h^*(x)$ and $g^*(x)$ are Fenchel duals of $h(x)$ and $g(x)$ respectively and the last equality is obtained using Fenchel duality theorem.
We note that with probability at least $1-n^{-c}$ $s_{\min}(A_S)\geq c$ and hence $A^T_SA_S$ is invertible and we have the following lemma:
\begin{lemma}{\label{Fen_dual}}
For the functions $h(x)$ and $b(x)$ defined above, we have the following Fenchel duals:
$$h^*(x)=\frac{1}{2}(x-A^T_Sy)^T(A^T_SA_S)^{-1}(x-A^T_Sy),$$
\begin{equation}
   b^*(x)=
   \begin{cases}
-\langle \mu_{0,S}, x\rangle \quad \text{ if } \|\tilde{\Omega}^{-1}_Sx\|_{\infty}\leq \lambda\\
\infty \text{ otherwise }.
\end{cases}
\end{equation}
\end{lemma}
It follows that
$$g(y)=\sup_{\|\tilde{\Omega}_S^{-1}x\|_{\infty}\leq \lambda}-\frac{1}{2}(x-A^T_Sy)^T(A^T_SA_S)^{-1}(x-A^T_Sy)-\langle \mu_{0,S}, x\rangle,$$
and consequently defining $P_S=A_S(A^T_SA_S)^{-1}A^T_S$  and using \eqref{f_bound} we obtain the following set of inequalities,
\begin{align*}&\frac{1}{2}\|A_{-S}(\hat{w}_{-S}-\hat{w}^{(S)}_{-S})\|^2\\
 \leq &|g(y_2)-g(y_1)|\\
 \leq &\sup_{\|\tilde{\Omega}_S^{-1}x\|_{\infty}\leq \lambda} |x^T(A^T_SA_S)^{-1}A^T_S(y_1-y_2)| + \frac{1}{2}|y_1^TP_Sy_1-y_2^TP_Sy_2|\\
 \leq& \sup_{\|\tilde{\Omega}_S^{-1}x\|_{\infty}\leq \lambda} \|x\|\|(A^T_SA_S)^{-1}A^T_S(y_1-y_2)\|+\frac{1}{2}(y_1-y_2)^TP_S(y_1-y_2)+|y_2^TP_S(y_1-y_2)|.
\end{align*}
Recall that $y_1-y_2=A_{-S}(\hat{w}_{-S}-\hat{w}^{(S)}_{-S})$; hence we obtain
$(y_1-y_2)^TP_S(y_1-y_2)=\|P_SA_{-S}(\hat{w}_{-S}-\hat{w}^{(S)}_{-S})\|^2$. For $\|\tilde{\Omega}^{-1}_Sx\|_{\infty}\leq \lambda$ we have the following inequality $$\|x\|\leq \|x\|_{\infty}\sqrt{d}\leq C\|\tilde{\Omega}^{-1}_Sx\|_{\infty}\sqrt{d}\leq C\lambda \sqrt{d}$$ and consequently we obtain,
\begin{align*}&\frac{1}{2}\|P^{\perp}_SA_{-S}(\hat{w}_{-S}-\hat{w}^{(S)}_{-S})\|^2\\
\leq &C\lambda \sqrt{d}\|(A^T_SA_S)^{-1}A^T_SA_{-S}(\hat{w}_{-S}-\hat{w}^{(S)}_{-S})\|+\|(A_{-S}\hat{w}^{(S)}_{-S}-\xi)^TP_SA_{-S}(\hat{w}_{-S}-\hat{w}^{(S)}_{-S})\|\\
\leq &C\lambda \sqrt{d}\|(A^T_SA_S)^{-1}A^T_SA_{-S}(\hat{w}_{-S}-\hat{w}^{(S)}_{-S})\|+\|(A_{-S}\hat{w}^{(S)}_{-S}-\xi)^TA_S(A^T_SA_S)^{-1}A^T_{S}A_{-S}(\hat{w}_{-S}-\hat{w}^{(S)}_{-S})\|\\
 \leq &(C\lambda \sqrt{d}\|(A^T_SA_S)^{-1}\|\|A^T_S\|\|A_{-S}\|+\|A^T_{-S}\|\|A_{S}\|\|(A^T_SA_S)^{-1}\|\|A^T_S(A_{-S}\hat{w}^{(S)}_{-S}-\xi)\|)\|\hat{w}_{-S}-\hat{w}^{(S)}_{-S}\|.
\end{align*}
By the same argument as equation (\ref{cross-term-bound}) we have $\|A^T_SA_{-S}\hat{w}^{(S)}_{-S}\|\leq C\sqrt{d \log n}$ with probability at least $1-n^{-c}$ and by independence of $A_S$ and $\xi$, $\|A^T_S\xi\|\leq C\sqrt{d \log n}$ with probability at least $1-n^{-c}$. Hence $\|A^T_S(A_{-S}\hat{w}^{(S)}_{-S}-\xi)\|\leq C\sqrt{d \log n}$ with probability at least $1-n^{-c}$.
By Lemma \ref{row_indep}  the terms $\|A_S\|,\|A_{-S}\|,\|(A^T_SA_S)^{-1}\|$ are bounded by some constant $C$ with probability at least $1-n^{-c}$. Hence the following holds 
\begin{equation}\frac{1}{2}\|P^{\perp}_SA_{-S}(\hat{w}_{-S}-\hat{w}^{(S)}_{-S})\|^2\leq C\sqrt{d\log n}\|\hat{w}_{-S}-\hat{w}^{(S)}_{-S})\|,\label{Lasso_bound}\end{equation}
with probability at least $1-n^{-c}$.
Now $$P_S^{\perp}(A_{-S}(\hat{w}_{-S}-\hat{w}_{-S}^{(S)}))=A_S(-(A^T_SA_S)^{-1}A^T_SA_{-S}(\hat{w}_{-S}-\hat{w}_{-S}^{(S)}))+A_{-S}(\hat{w}_{-S}-\hat{w}_{-S}^{(S)})=A\bar{w},$$
where 
$$\bar{w}^T=((-(A^T_SA_S)^{-1}A^T_SA_{-S}(\hat{w}_{-S}-\hat{w}_{-S}^{(S)}))^T,(\hat{w}_{-S}-\hat{w}_{-S}^{(S)})^T).$$
By Lemma \ref{row_indep} and the condition $n/p> C^2_K$ we have
$\|A\bar{w}\|^2\geq c\|\bar{w}\|^2$ (for some $c>0$).  We also note that $\|\bar{w}\|^2\geq \|\hat{w}_{-S}-\hat{w}_{-S}^{(S)}\|^2$ and hence,
\begin{align}c\|\hat{w}_{-S}-\hat{w}_{-S}^{(S)}\|^2&\leq C\sqrt{d\log n}\|\hat{w}_{-S}-\hat{w}^{(S)}_{-S})\|\nonumber,\\
\|\hat{w}_{-S}-\hat{w}_{-S}^{(S)}\|&\leq c^{-1}C\sqrt{d\log n}\label{diff_bound}.
\end{align}
By the KKT condition
$$A^TAw=A^T\xi-\lambda \tilde{\Omega} v,$$ for subgradient $v$. Arguing that $\|A^T\xi\|_{\infty}$ and $\lambda \|\tilde{\Omega} v\|_{\infty}$ are bounded by $C\sqrt{\log n}$ with probability at least $1-n^{-c}$ we obtain
$$\|(A^TA\hat{w})_S\|=\|A^T_SA_S\hat{w}_S+A_S^TA_{-S}\hat{w}_{-S}\|\leq C\sqrt{d\log n}.$$
Rearranging terms and using triangle inequality then yields,
$$\|A^T_SA_S\hat{w}_S\|\leq C(\sqrt{d\log p}+\|A_S^TA_{-S}\hat{w}_{-S}\|).$$
Noting that $\|A^T_SA_S\|\geq c$ holds with probability at least $1-n^{-c}$ we have the following decomposition:
$$\|\hat{w}_S\|\leq C(\|A_S^TA_{-S}\hat{w}_{-S}\|+\sqrt{d\log n})\leq C(\|A_S^TA_{-S}\hat{w}^{(S)}_{-S}\|+\|A_S\|\|A_{-S}\|\|\hat{w}_{-S}-\hat{w}_{-S}^{(S)}\|+\sqrt{d\log n}).$$
The term $\|A_S^TA_{-S}\hat{w}_{-S}^{(S)}\|$ can be bounded by $\sqrt{d\log n}$ by the same argument as equation \ref{cross-term-bound}. Also $\|A_{S}\|,\|A_{-S}\|$ are bounded by some constant $C$ with probability at least $1-n^{-c}$ by applying Lemma \ref{row_indep}. These facts together with \eqref{diff_bound} yields
$$|\hat{w}|_s\leq \|\hat{w}_S\|\leq \sqrt{d\log n},$$ with probability at least $1-n^{-c}$ and using a union bound we obtain
$$\|\hat{w}^L\|_{\infty}\leq \sqrt{d \log n},$$ with probability at least $1-n^{-c}$.\\

\textbf{Case 2:} $n/p \rightarrow \kappa^{-1} \leq C^2_K$\\

We note that for the $\ell_{\infty}$ bound in the ridge regression the lower bound of the singular value from Lemma \ref{row_indep} was not used, hence in this case the $\ell_{\infty}$ bound proof is identical to the one used in Case 1.

For the $\ell_{\infty}$ in Lasso we note the lower bound on the singular from Lemma \ref{row_indep} was only used to give lower bound to $A\bar{w}$ (see the equation after \ref{Lasso_bound}). However in this case one can use easily show that for $\lambda\geq\lambda_0$ $\hat{w}_{-S}-\hat{w}^{(S)}_S$ and hence $\bar{w}$ is $2c_0n$ sparse (see Lemma \ref{Lasso_sparse}). The rest of the proof is same as the Lasso case of Lemma \ref{ell_infty_bound1}.

\end{proof}

\begin{proof} [Proof of Lemma \ref{Lasso_sparse}]
    This follows from the proof of Lemma 6.3 of \cite{han2023universality} with certain modification due to the dependence. For the sake of completeness we give the entire proof.

From the KKT condition we can write 
\begin{equation}A_i^T(A\hat{w}-\xi)+\lambda (\tilde{\Omega} v)_i =0,\label{sparse_lasso_kkt}\end{equation}
with $v$ be the subgradient of $\|.\|_1$ at $\tilde{\Omega}(\hat{w}+\mu_0)$. Let $S_+$ be the set such that for $i\in S_+$, $\hat{\mu}_i=(\hat{w}+\mu_0)_i>0$ (and similarly define $S_-$). We will show that both $S_+$ and $S_-$ are bounded by $c_0n$ with probability at least $1-n^{-c}$ which ensures that the Lasso solution is $2c_0n$ sparse with probability at least $1-n^{-c}$.

Note that the KKT condition \eqref{sparse_lasso_kkt} can be used to show 
\begin{align*}\sum_{i\in S_+}z^2_i+\lambda^2\sum_{i\in S_+}\omega^2_i-2\lambda\sum_{i\in S_+}z_i\omega_i&\leq \|A\hat{w}\|^2\|A_{S_+}A^T_{S_+}\|
\\
c^2\lambda^2|S_+|-2\lambda\sum_{i\in S_+}z_i\omega_i&\leq \|A\hat{w}\|^2\|A_{S_+}A^T_{S_+}\|,
\end{align*}
where $z=A^T\xi$ and we have used the lower bound on $\omega_i$. Thus with probability at least $1-n^{-c}$ the following event holds
$$E:=\{c^2\lambda^2|S_+|\leq 2\lambda|\sum_{i\in S_+}z_i\omega_i|+Cn\}.$$
We have used the fact that $\|A\hat{w}\|\leq\sqrt{n}$ with probability at least $1-n^{-c}$ and $\|A_{S_+}A^T_{S_+}\|\leq C$
with the same probability.
If $|S_+|\geq c_0n$, we have $c^2\lambda^2|S_+|\leq 2\lambda|\sum_{i\in S_+}z_i\omega_i|+(C/c_0)|S_+|$ so that
if $\lambda$ is chosen greater than $\lambda_0=(2C/c^2c_0)^{1/2}$, we have
$$c^2\lambda^2|S_+|/2\leq 2\lambda|\sum_{i\in S_+}z_i\omega_i|.$$
To show that this event hold with vanishing probability we first write $\sum_{i\in S_+}z_i\omega_i=\xi^TAe_{s}$ where $e_{s}$ is the vector whose $j^{th}$ entry is $\omega_j\mathds{1}\{j\in S_+\}$. We observe that $(Ae_s)_i$ is subGaussian with subGaussian factor $Cc^2s/n$ where $s=|S_+|$, since $A=X/\sqrt{n}$ and $X$ has isotropic subGaussian rows. Hence conditioned on $\xi$ the random variable $\xi^TAe_{s}$ has subGaussian factor $Cc^2\|\xi\|^2s/n$ since $\xi$ is independent of $A$. We define the event $F:=\{\|\xi\|\leq C\sqrt{n}\}$.  Since $\|\xi\|\leq C\sqrt{n}$ with probability $1-n^{-c}$, for a given subset $S$ of size s we have
$$P(\{c^2\lambda^2s/2\leq 2\lambda|\sum_{i\in S_+}z_i\omega_i|\}\cap F)\leq \exp(-c^4\lambda^2(s/4)^2/2Cs)$$
where we have simplified the constants. Finally we obtain,
\begin{align*}
P(\{|S_+|\geq c_0n\}\cap E\cap F)&\leq P(\{\lambda c^2|S_+|/4\leq |\sum_{i\in S_+}z_i\omega_i|\}\cap F).\\
& \leq\sum_{s=c_0n}^p\sum_{S:|S|=s}\exp(-c^4\lambda^2(s/4)^2/2Cs)\\
&\leq C\exp(-Cc_0\lambda^2n),
\end{align*}
where we have simplified all the constants except $c_0$. We can remove the conditioning on $E$ and $F$ since both holds with probability at least $1-n^{-c}$.

Similar bounds hold for $S_-$ and we have shown (redefining $c_0$) that for $\lambda\geq\lambda_0$, the Lasso solution is $c_0n$ sparse with probability at least $1-n^{-c}$.
\end{proof}

\begin{proof}[Proof of Lemma \ref{Fen_dual}]
We start with the Fenchel dual for the function $h$,
    $$h^*(x)=\sup_{z}\langle z,x\rangle-\frac{1}{2}\|A_Sz\|^2-\langle A_Sz, y\rangle.$$
It can be easily verified that the maximum is attained at $z=(A_S^TA_S)^{-1}(x-A_S^Ty)$ and plugging in the value we obtain
$$h^*(x)=\frac{1}{2}(x-A^T_Sy)^T(A^T_SA_S)^{-1}(x-A^T_Sy).$$
To compute $b^*(x)$ we start with defining the function $\bar{b}(x)=\lambda \omega|x+y|$ for scalars $x$ and $y$. Thus we have,
$$\bar{b}^*(x)=\sup_{z}xz-\lambda \omega|z+y|.$$
It can be easily observed that $\bar{b}^*(x)=\infty$ unless $|x|\leq \lambda \omega$.

Further define the function
$$\alpha(x,z)=xz-\lambda \omega|z+y|.$$
If $z\geq -y$, the function $\alpha(x,z)=z(x-\lambda\omega)-\lambda y$
attains maximum value at $z=-y$
as $x-\lambda\omega\leq 0$. Similarly if $z\leq -y$, $\alpha(x,z)=z(x+\lambda\omega)+\lambda y$ attains maximum value at $z=-y$
as $x+\lambda\omega\geq 0$. Hence 
\begin{equation*}
   \bar{b}^*(x)=
   \begin{cases}
\sup_{z}\alpha(x,z)=-xy, \quad \text{ if } |x|\leq \lambda \omega\\
\infty, \text{ otherwise }.
\end{cases}
\end{equation*}
Now we consider the function
$$b^*(x)=\sup_{z}\langle z,x\rangle-\lambda\|\tilde{\Omega}_S(z+\mu_{0,S})\|_1$$
and note that the optimization problem in the RHS is separable in its arguments. Hence we have
\begin{equation*}
   b^*(x)=
   \begin{cases}
-\langle\mu_{0,S},x\rangle, \quad \text{ if } |x_i/\omega_i|\leq \lambda \text{ for all } i\in S \\
\infty, \text{ otherwise },
\end{cases}
\end{equation*}
which was to be shown. 
\end{proof}

\subsection{Universality of the Optimizer \label{univ_opt_appen}}

\begin{proof}[Proof of Lemma \ref{Gordon_gap}]

We only prove the result for the ridge case, since the Lasso case follows by suitably modifying (to account for the term $\tilde{\Omega}$) the proof of Theorem 3.8 of \cite{han2023universality} together with the statements from \cite{miolane_mont}. Recall that $G$ is the design matrix with i.i.d.~Gaussian entries. Denoting $\bar{G}=G/\sqrt{n}$ problem, we consider the following functions:
\begin{align*}r(w,u)&=\frac{1}{n}u^T\bar{G}w-\frac{1}{n}u^T\xi-\frac{1}{2n}\|u\|^2+\frac{\lambda}{2n}(\|\tilde{\Omega}(w+\mu_0)\|^2-\|\tilde{\Omega}\mu_0\|^2),\\
\ell(w,u)&=-\frac{1}{n^{-3/2}}\|u\|g^Tw+\frac{\|w\|h^Tu}{n^{3/2}}-\frac{1}{n}u^T\xi-\frac{1}{2n}\|u\|^2+\frac{\lambda}{2n}(\|\tilde{\Omega}(w+\mu_0)\|^2-\|\tilde{\Omega}\mu_0\|^2),
\end{align*}
where $u\in \mathbb{R}^n$ and $h$ and $g$ are independent random vectors with distributions $h\sim N(0,I_n)$, $g\sim N(0,I_p)$ respectively. We denote 
 by $R(w)=\max_u h(w,u)$ and $L(w)=\max_u \ell(w,u)$ the ridge cost and the associated Gordon cost respectively.

Note that using $w=\sqrt{n}\theta$, $R(w)=R^K(\theta,G)$ for $K=R$ and we work with the notation $R(w)$ since the Gaussian design and the ridge setup is clear from the context. Our goal is to demonstrate that the optimal value of $R(w)$ i.e. $\min_{w}R(w)$ is close to $\psi^R$. We will use CGMT to show that $\min_{w}R(w)$ is well approximated by $\min_{w}L(w)$ and the latter is close to the desired constant $\psi^R$ (a constant that we will define shortly). The proof proceeds by showing that  $\min_{w}L(w)$ is well approximated by the value at saddle point of a function $\psi_n(\beta,\gamma)$, which is further approximated by the value at saddle point of its non-stochastic version $\psi(\beta,\gamma)$. Denoting the saddle point of the latter by $(\beta_*,\gamma_*)$, $\psi^R$ will be defined as $\psi(\beta_*,\gamma_*)$. We start by defining $\psi_n$ and $\psi$.

Define the function $\psi_n(\beta,\gamma)$ as follows:
\begin{equation}
    \psi_n(\beta,\gamma)=\left(\frac{\sigma^2}{\gamma}+\gamma\right)\frac{\beta}{2}-\frac{\beta^2}{2}-\frac{1}{n/p}\frac{1}{p}\sum_{i=1}^p\frac{\gamma \lambda^2\omega^4_i}{2(\beta+\lambda \omega^2_i\gamma)}\left(\mu_
    {0,i}-\frac{\beta g_i}{\lambda \omega^2_i}\right)^2.\label{emp_psi}
\end{equation}
The above equation can be written compactly as
\begin{equation}
    \psi_n(\beta,\gamma)=\left(\frac{\sigma^2}{\gamma}+\gamma\right)\frac{\beta}{2}-\frac{\beta^2}{2}-\frac{1}{n/p}E_{P_n}\left[\frac{\gamma \lambda^2\Omega^4}{2(\beta+\lambda \Omega^2\gamma)}\left(M-\frac{\beta \Gamma}{\lambda \Omega^2}\right)^2\right],\label{emp_psi_1}
\end{equation}
where the expectation is taken with respect to the law $P_n$ of the $3$-tuple of random variables $(M,\Omega,\Gamma)\sim \frac{1}{p}\sum_{i=1}^p\delta_{\mu_i,\omega_i,g_i}$.

Also define the function $\psi(\beta,\gamma)$ as follows:

\begin{equation}
    \psi(\beta,\gamma)=\left(\frac{\sigma^2}{\gamma}+\gamma\right)\frac{\beta}{2}-\frac{\beta^2}{2}-\frac{1}{n/p}E_{Q_n}\left[\frac{\gamma \lambda^2\Omega^4}{2(\beta+\lambda \Omega^2\gamma)}\left(M-\frac{\beta \Gamma}{\lambda \Omega^2}\right)^2\right],\label{emp_psi_2}
\end{equation}
where the expectation is taken with respect to the law $Q_n$ of the $3$-tuple of random variables $(M,\Omega,\Gamma)\sim (\frac{1}{p}\sum_{i=1}^p\delta_{\mu_{0,i},\omega_i})\otimes N(0,1)$.
Taking the derivates of the function $\psi_n$ with respect to $\beta$ and $\gamma$ one obtains the following fixed point equations for the saddle points:
\begin{align}
\gamma^2_{*,n}&=\sigma^2+\frac{1}{n/p}E_{P_n}\left[\eta_2\left(M+\gamma_{*,n}\Gamma;\frac{\gamma_{*,n}\lambda\Omega^2}{\beta_{*,n}}\right)-M\right]^2\label{emp_fixed_pt_1},\\
\beta_{*,n}&=\gamma_{*,n}-\frac{1}{n/p}E_{P_n}\left[\Gamma.\left(\eta_2\left(M+\gamma_{*,n}\Gamma;\frac{\gamma_{*,n}\lambda\Omega^2}{\beta_{*,n}}\right)-M\right)\right].\label{emp_fixed_ot_2}
\end{align}
Similarly 
 from equation \eqref{emp_psi_2} we get the fixed point equations for the limiting case as follows:
\begin{align}
\gamma^2_{*}&=\sigma^2+\frac{1}{n/p}E_{Q_n}\left[\eta_2\left(M+\gamma_{*}\Gamma;\frac{\gamma_{*}\lambda\Omega^2}{\beta_{*}}\right)-M\right]^2,\label{limit_fixed_pt_1}\\
\beta_{*}&=\gamma_{*}-\frac{1}{n/p}E_{Q_n}\left[\Gamma.\left(\eta_2\left(M+\gamma_{*}\Gamma;\frac{\gamma_{*}\lambda\Omega^2}{\beta_{*}}\right)-M\right)\right]\label{limit_fixed_pt_2}.
\end{align}
It can be shown following the steps of 
 Proposition 5.2 of \cite{han2023universality} that the  solutions $\beta_{*,n}$ and $\gamma_{*,n}$ lie in some compact subset of $\mathbb{R}^+$ bounded away from $0$ with probability at least $1-n^{-c}$ and the same can be shown for $\beta_{*}$ and $\gamma_{*})$. Next we show that $|\beta_{*,n}-\beta_{*}|$ and $|\gamma_{*,n}-\gamma_{*}|$ converges to $0$ with probability at least $1-n^{-c}$. However because of the matrix $\Omega$ this convergence does not follow easily from \cite{han2023universality}. We use different techniques and prove the statement in a separate lemma whose proof is technical and hence we defer for the moment

 \begin{lemma}{\label{convergence_of_fixed_points}}
 With probability at least $1-n^{-c}$, we have
$$|\beta_{*,n}-\beta_{*}|=O(\epsilon_n) \quad |\gamma_{*,n}-\gamma_{*}|=O(\epsilon_n).$$
for some $\epsilon_n\rightarrow 0$.
\end{lemma}
The rest of the proof follows in the lines of \cite{han2023universality}. Since the modification involving the matrix $\Omega$ is routine we only sketch the outline and omit detailed proof. One shows (see \cite[Lemma 5.2]{han2023universality}) that with probability at least $1-n^{-c}$ the following statements hold:
\begin{enumerate}
\item$|\psi_n(\beta_{*,n},\gamma_{*,n})-\psi(\beta_{*},\gamma_{*})|=O(\epsilon_n)$,
\item $|\min_{w}L(w)-\max_{\beta>0}\min_{\gamma>0}\psi_n(\beta,\gamma)|=O(\epsilon_n)$,
\item $|\min_{w}H(w)-\psi(\beta_*,\gamma_*)|=O(\epsilon_n).$
\end{enumerate}
Now for any Lipshitz function $\varrho$ define the following set
$$S_{n,\epsilon}(\varrho)=\{w\in \mathbb{R}^p:|\varrho(w/\sqrt{n})-E\varrho(w_*/\sqrt{n})|\geq \epsilon^{-1/2}\}.$$ 
Then it can be shown (see \cite[Theorem 3.4]{han2023universality} and its proof) that the following statement holds:
\begin{align*}
    P\left(\min_w R(w)\geq \psi(\beta_*,\gamma_*)+\epsilon\right)&\leq Cn^{-c},\\
    P\left(\min_{w \in S_{n,\epsilon}(\varrho)} R(w)\leq \psi(\beta_*,\gamma_*)+K^{-1}\epsilon\right)&\leq Cn^{-c}
\end{align*}
for some $K>0$. Recall that $\theta=w/\sqrt{n}$, $\psi^R=\psi(\beta_*,\gamma_*)$. Then letting $\delta_n=Cn^{-c}$ and choosing $z$ suitably the above inequalities imply that the statement of the  Lemma \ref{Gordon_gap} is true if we can prove Lemma \ref{convergence_of_fixed_points}.
 \end{proof}
 
 \begin{proof}[Proof of lemma \ref{convergence_of_fixed_points}]
 
 First we  adopt the method used in \cite{montanari2019generalization} to show that with probability at least $1-n^{-c}$
 \begin{equation}\|(\beta_{*,n},\gamma_{*,n})-(\beta^{\dagger},\gamma^{\dagger})\|=O(\epsilon_n)\label{P_n_conv}
 \end{equation}
 holds for some $\epsilon_n\rightarrow 0$ and for some constants $\beta^{\dagger},\gamma^{\dagger}$. Let us denote the joint law of the random variables $(\bar{M},\bar{\Omega},Z)$ as $P$, where $(\bar{M},\bar{\Omega})$ has law $\mu$ (see Assumption \ref{ass_4}) and $Z$ is independently distributed as  $N(0,1)$. Note that by Assumption 4 $P_n\xRightarrow{W_2}P$ and $Q_n\xRightarrow{W_2}P$.
For notational convenience we will define functions of a general 3-touple of random variables $(M,\Omega,\Gamma)$  and depending on the context they will be distributed as $P_n, Q_n$ or $P$.\\

\textbf{Step:1} We write the optimizers as solutions of some functions i.e. define
\begin{align}F_n(\beta,\gamma)&=\frac{1}{n/p}\left[\eta_2\left(M+\gamma \Gamma;\frac{\gamma\lambda\Omega^2}{\beta}\right)-M\right]^2-\sigma^2-\gamma^2,\\
 G_n(\beta,\gamma)&=\frac{1}{n/p}\left[\Gamma.\left(\eta_2\left(M+\gamma \Gamma;\frac{\gamma\lambda\Omega^2}{\beta}\right)-M\right)\right]+\beta-\gamma,\\
F(\beta,\gamma)&=\kappa\left[\eta_2\left(M+\gamma \Gamma;\frac{\gamma\lambda\Omega^2}{\beta}\right)-M\right]^2-\sigma^2-\gamma^2,\\
G(\beta,\gamma)&=\kappa\left[\Gamma.\left(\eta_2\left(M+\gamma\Gamma;\frac{\gamma\lambda\Omega^2}{\beta}\right)-M\right)\right]+\beta-\gamma.\end{align}
Note that $(\beta_{*,n},\gamma_{*,n})$ are solutions of the equations $E_{P_n}F_n(\beta,\gamma)$ and $E_{P_n}G_n(\beta,\gamma)$. We define $(\beta^{\dagger},\gamma^{\dagger})$ as solutions of the equation $E_{P}F(\beta,\gamma)$ and $E_{P}G(\beta,\gamma)$. It can be shown along the lines of Lemma 5.2 of \cite{han2023universality} that both solutions are unique.\\

\textbf{Step 2:} We will show that the following uniform convergence holds in some compact set $\mathcal{C}$ of $\mathbb{R}^+\times \mathbb{R}^+$bounded away from $(0,0)$:
\begin{align}\lim_{n\rightarrow\infty}\sup_{(\gamma,\beta)\in \mathcal{C}}|E_{P_n}F_n(\beta,\gamma)-E_PF(\beta,\gamma)|&=0\label{unif_conv_1},\\
\lim_{n\rightarrow\infty}\sup_{(\gamma,\beta)\in \mathcal{C}}|E_{P_n}G_n(\beta,\gamma)-E_PG(\beta,\gamma)|&=0\label{unif_conv_2}.
 \end{align}
We defer the proof for the moment.\\

\textbf{Step 3:} 
 Since the optimizers in both the finite sample case ($(\beta_{*,n},\gamma_{*,n})$) and the limiting case ($(\beta^{\dagger},\gamma^{\dagger})$) lie in compact sets with probability at least $1-n^{-c}$ (see Lemma 5.2 of \cite{han2023universality}), we can choose $\mathcal{C}$, so that the optimizers belong to $\mathcal{C}$ with probability at least $1-n^{-c}$. Now let $(\tilde{\beta},\tilde{\gamma})$ be a limit point of the sequence $(\beta_{*,n},\gamma_{*,n})$ then we prove that $(\tilde{\beta},\tilde{\gamma})=(\beta^{\dagger},\gamma^{\dagger})$.

 Indeed,
$$|E_P(F(\tilde{\beta},\tilde{\gamma}))|\leq |E_P(F(\tilde{\beta},\tilde{\gamma}))-E_P(F(\beta_{*,n},\gamma_{*,n}))|+|E_{P_n}(F(\beta_{*,n},\gamma_{*,n}))-E_P(F(\beta_{*,n},\gamma_{*,n}))|,$$
since $E_{P_n}(F(\beta_{*,n},\gamma_{*,n}))=0$. Now from the continuity of $E_P(F(.,.))$ we have $$|E_P(F(\tilde{\beta},\tilde{\gamma}))-E_P(F(\beta_{*,n},\gamma_{*,n}))|\rightarrow 0,$$ while from equation \eqref{unif_conv_1} we conclude that $$|E_{P_n}(F(\beta_{*,n},\gamma_{*,n}))-E_P(F(\beta_{*,n},\gamma_{*,n}))|\rightarrow 0.$$
Thus $E_P(F(\tilde{\beta},\tilde{\gamma}))=0$ and similarly we can prove $E_P(G(\tilde{\beta},\tilde{\gamma}))=0$. Since $(\beta^{\dagger},\gamma^{\dagger})$ is the unique solution of the set of equations $E_P(F(\beta,\gamma))=0$ and $E_P(G(\beta,\gamma))=0$, we obtain $(\tilde{\beta},\tilde{\gamma})=(\beta^{\dagger},\gamma^{\dagger})$.
 Noting that the optimizers $(\beta_{*n},\gamma_{*n})$ and $(\beta^{\dagger},\gamma^{\dagger})$ lie in the compact set $\mathcal{C}$ with probability at least $1-n^{-c}$ 
 this proves that for some $\epsilon_n$, we have $\|(\beta_{*,n},\gamma_{*,n})-(\beta^{\dagger},\gamma^{\dagger})\|=O(\epsilon_n)$  with probability at least $1-n^{-c}$.

Thus, to prove the lemma it is sufficient to prove \eqref{unif_conv_1} and \eqref{unif_conv_2}. 
Since $p/n\rightarrow \kappa$ it is enough to prove $$\lim_{n\rightarrow\infty}\sup_{(\gamma,\beta)\in\mathcal{C}}|E_{P_n}F(\beta,\gamma)-E_PF(\beta,\gamma)|=0.$$ By Arzela-Ascoli theorem the proof follows if we can show the following statements:
\begin{enumerate}
    \item Pointwise convergence: For a given $(\gamma,\beta)$ in $\mathcal{C}$, $\lim_{n\rightarrow\infty}|E_{P_n}F(\beta,\gamma)-E_PF(\beta,\gamma)|=0.$
    \item Equicontinuity:
    $\forall 
 \epsilon, \exists \delta$  such that if $\|(\beta,\gamma)-(\beta',\gamma')\|<\delta$, 
$$\sup_n|E_{P_n}F(\beta,\gamma)-E_{P_n}F(\beta',\gamma')|<\epsilon.$$
\end{enumerate}
Writing $F(\beta,\gamma)$ as $f(M,\Gamma,\Omega)$, a function of $M,\Gamma,\Omega$, one observes that to show (1), the following result can be used: $P_n\xRightarrow{W_2}P$ implies $|E_{P_n}F(\beta,\gamma)-E_PF(\beta,\gamma)|=0$, if one can show
$\sup_{(M,\Gamma,\Omega)}\frac{f(M,\Gamma,\Omega)}{\|M,\Gamma,\Omega\|^2}<\infty$.
Noting the form of $F(\beta,\gamma)$ it is enough to show that,
$$\sup_{(M,\Gamma,\Omega)}\frac{\left[\eta_2\left(M+\gamma \Gamma;\frac{\gamma\lambda\Omega^2}{\beta}\right)-M\right]^2}{\|M,\Gamma,\Omega\|^2}\leq\sup_{(M,\Gamma,\Omega)}\frac{\left[(-\lambda\Omega^2M+\beta \Gamma)/(\beta\gamma^{-1}+\lambda\Omega^2)\right]^2}{\|M,\Gamma\|^2}<\infty.$$
Since $\Omega$ is bounded, the numerator is bounded by a quadratic function of $M$ and $G$ and hence the claim is proved.

For the function $G$, we similarly have to show
$$\sup_{(M,\Gamma,\Omega)}\frac{\Gamma(\eta_2\left(M+\gamma \Gamma;\frac{\gamma^2\lambda\Omega^2}{\beta}\right)-M)}{\|M,\Gamma,\Omega\|^2}\leq\sup_{(M,\Gamma,\Omega)}\frac{\Gamma(-\lambda\Omega^2M+\beta \Gamma)/(\beta\gamma^{-1}+\lambda\Omega^2)}{\|M,\Gamma\|^2}<\infty.$$
Again, the numerator is upper bounded by a quadratic function of $\Gamma$ and $M$ is bounded; hence, the claim is proved.

To show (2) we observe that,
\begin{align*}F(\beta,\gamma)-F(\beta',\gamma')&=\left[(-\lambda\Omega^2M+\beta \Gamma)/(\beta\gamma^{-1}+\lambda\Omega^2)\right]^2-\left[(-\lambda\Omega^2M+\beta' \Gamma)/(\beta'\gamma'^{-1}+\lambda\Omega^2)\right]^2+\gamma^{'2}-\gamma^{2}\\
&=\frac{(\beta'\gamma'^{-1}+\lambda\Omega^2)^2(-\lambda\Omega^2M+\beta \Gamma)^2-(\beta\gamma^{-1}+\lambda\Omega^2)^2(-\lambda\Omega^2M+\beta' \Gamma)^2}{(\beta\gamma^{-1}+\lambda\Omega^2)^2(\beta'\gamma'^{-1}+\lambda\Omega^2)^2}+\gamma^{'2}-\gamma^{2}.
\end{align*}
Noting that $\beta,\beta',\gamma,\gamma'$ lies in compact sets bounded away from $0$, we need to show that $\|(\beta,\gamma)-(\beta',\gamma')\|<\delta$ implies
\begin{equation}|E_{P_n}[(\beta'\gamma'^{-1}+\lambda\Omega^2)^2(-\lambda\Omega^2M+\beta \Gamma)^2-(\beta\gamma^{-1}+\lambda\Omega^2)^2(-\lambda\Omega^2M+\beta' \Gamma)^2]|<\epsilon. \label{equicont_mid}\end{equation}
First, we show that
\begin{equation}|E_{P_n}[(\beta'\gamma'^{-1}+\lambda\Omega^2)^2((-\lambda\Omega^2M+\beta \Gamma)^2-(-\lambda\Omega^2M+\beta' \Gamma)^2)]|<\epsilon/2.\label{equicont_mid1}\end{equation}
 Indeed since $(\beta'\gamma'^{-1}+\lambda\Omega^2)^2$ is bounded by some constant $C$, we consider the expression the following set of inequalities:
\begin{align*}
    |E_{P_n}[((-\lambda\Omega^2M+\beta \Gamma)^2-(-\lambda\Omega^2M+\beta' \Gamma)^2)]|&\leq E_{P_n}[|\beta^{'2}-\beta^{2}| \Gamma^2]+E_{P_n}[2|\beta-\beta'|\lambda\Omega^2M \Gamma]|\\
    &\leq E_{P_n}[|\beta^{'2}-\beta^{2}| \Gamma^2]+CE_{P_n}[|\beta-\beta'|\lambda(M^2+\Gamma^2)]|.
\end{align*}
Since both $x^2$ and $x$ are uniformly continuous functions in the compact domain and $P_n\xRightarrow{W_2}P$ implies that
$E_{P_n}\Gamma^2$ and $E_{P_n}(M^2+\Gamma^2)$ is uniformly bounded we can choose $\delta$ small enough so that equation \eqref{equicont_mid1} holds.

Next we show that
\begin{equation}|E_{P_n}[((\beta'\gamma'^{-1}+\lambda\Omega^2)^2-(\beta\gamma^{-1}+\lambda\Omega^2)^2)(-\lambda\Omega^2M+\beta' \Gamma)^2)]|<\epsilon/2.\label{equicont_mid2}\end{equation}
Using the facts that $\Omega$ is bounded, $\beta\gamma^{-1}$ and $\beta^2\gamma^{-2}$ are uniformly continuous functions in the compact domain (the domain also being bounded away from $0$), we have  $((\beta'\gamma'^{-1}+\lambda\Omega^2)^2-(\beta\gamma^{-1}+\lambda\Omega^2)^2)<\varepsilon$ for sufficiently small $\delta$. Also, $P_n\xRightarrow{W_2}P$ implies that $E_{P_n}(-\lambda\Omega^2M+\beta' \Gamma)^2$ is uniformly bounded in $n$, so that choosing $\varepsilon$ small enough, equation \eqref{equicont_mid2} holds. Finally we use triangle inequality to show \eqref{equicont_mid} from \eqref{equicont_mid1} and \eqref{equicont_mid2}. The equicontinuity of $G$ can be shown along similar lines; hence we omit the proof.\\

\textbf{Step 4:} A calculation analogous to the one shown above proves that the following statements hold:
\begin{align}\lim_n\sup_{(\gamma,\beta)\in \mathcal{C}}|E_{Q_n}F_n(\beta,\gamma)-E_PF(\beta,\gamma)|&=0,\label{unif_conv_3}\\
\lim_n\sup_{(\gamma,\beta)\in \mathcal{C}}|E_{Q_n}G_n(\beta,\gamma)-E_PG(\beta,\gamma)|&=0\label{unif_conv_4}.
 \end{align}
An argument similar to the one used in Step 3 can be used to conclude that with probability at least $1-n^{-c}$ the following holds
\begin{equation}\|(\beta_{*},\gamma_{*})-(\beta^{\dagger},\gamma^{\dagger})\|=O(\epsilon_n)\label{Q_n_conv}\end{equation}
Using equations \eqref{P_n_conv} and \eqref{Q_n_conv} we conclude that the desired result holds.
\end{proof}

\end{document}